\documentclass[11pt,a4paper]{amsart}
\usepackage{amsmath}
\usepackage{amssymb}
\usepackage{amsthm}
\usepackage{amsfonts}
\usepackage[left=3cm,top=3cm,right=3cm]{geometry}

\usepackage[latin1]{inputenc}
\usepackage[english]{babel}

\newtheorem{thm}{Theorem}[section]
\newtheorem{lemma}[thm]{Lemma}

\newtheorem{coro}[thm]{Corollary}
\newtheorem{proposition}[thm]{Proposition}

\theoremstyle{remark}           
\newtheorem{remark}[thm]{Remark}
\newtheorem{example}[thm]{Example}

\newcommand{\diam}{\operatorname{diam}}

\newcommand{\dist}{\operatorname{dist}}

\newcommand{\rad}{\operatorname{rad}}
\newcommand{\spt}{\operatorname{spt}}

\newcommand{\bdry}{\partial}
\newcommand{\R}{\mathbb R}
\newcommand{\Om}{\Omega}
\newcommand{\N}{\mathbb N}
\newcommand{\Z}{\mathbb Z}

\newcommand{\W}{\mathcal W}

\newcommand{\Ha}{\mathcal H}

\newcommand{\B}{\mathcal B}

\newcommand{\sub}{\subset}
\newcommand{\eps}{\varepsilon}
\newcommand{\vphi}{\varphi}
\newcommand{\wtilde}{\widetilde}

\newcommand{\emp}{\emptyset}
\newcommand{\ol}{\overline}

\newcommand{\Char}[1]{\chi_{\lower 1.5pt\hbox{$\scriptscriptstyle #1$}}}
\newcommand{\dom}{{d_{\Omega}}}

\newcommand{\lQ}{\underline{Q}}
\newcommand{\uQ}{\overline{Q}}

\def\vint{\mathop{\mathchoice%
          {\setbox0\hbox{$\displaystyle\intop$}\kern 0.22\wd0%
           \vcenter{\hrule width 0.6\wd0}\kern -0.82\wd0}%
          {\setbox0\hbox{$\textstyle\intop$}\kern 0.2\wd0%
           \vcenter{\hrule width 0.6\wd0}\kern -0.8\wd0}%
          {\setbox0\hbox{$\scriptstyle\intop$}\kern 0.2\wd0%
           \vcenter{\hrule width 0.6\wd0}\kern -0.8\wd0}%
          {\setbox0\hbox{$\scriptscriptstyle\intop$}\kern 0.2\wd0%
           \vcenter{\hrule width 0.6\wd0}\kern -0.8\wd0}}%
          \mathopen{}\int}

\newcommand{\capp}[1]{\operatorname{cap}_{#1}}

\DeclareMathOperator{\dima}{dim_A}
\DeclareMathOperator{\udima}{\overline{dim}_A}
\DeclareMathOperator{\ldima}{\underline{dim}_A}
\DeclareMathOperator{\dimh}{dim_H}

\DeclareMathOperator{\ucodima}{\overline{co\,dim}_A}
\DeclareMathOperator{\lcodima}{\underline{co\,dim}_A}
\DeclareMathOperator{\codimh}{co\,dim_H}

\DeclareMathOperator{\udimm}{\overline{dim}_M}
\DeclareMathOperator{\ldimm}{\underline{dim}_M}

\newcommand{\Lip}{\operatorname{Lip}}
\newcommand{\lip}{\operatorname{lip}}

\begin{document}

\title{Hardy inequalities and Assouad dimensions}
\author{Juha Lehrb\"ack} 

\address{University of Jyvaskyla,
Department of Mathematics and Statistics, 
P.O. Box 35 (MaD), 
FIN-40014 University of Jyvaskyla, 
Finland}
\email{\tt juha.lehrback@jyu.fi}

\thanks{The author has been supported by the Academy of Finland,  grant no.\ 252108.} 


\begin{abstract}
We establish both sufficient and necessary conditions for weighted Hardy inequalities in 
metric spaces in terms of Assouad (co)dimensions. Our sufficient conditions in the case
where the complement is thin are new even in Euclidean spaces, while in the case of
a thick complement we give new formulations for previously known sufficient
conditions which reveal a natural duality between these two cases. 
Our necessary conditions
are rather straight-forward generalizations from the unweighted case,
but together with some examples they
indicate the essential sharpness of our results. 
In addition, we consider the
mixed case where the complement may contain both thick and thin parts.
\end{abstract}

\maketitle

\section{Introduction}

Let $X$ be a complete metric measure space.
We say that an open set $\Omega\sub X$ \emph{admits a $(p,\beta)$-Hardy inequality},
if there exists a constant $C>0$ such that the inequality
\begin{equation*}\label{eq:hardy*}
\int_{\Omega} |u(x)|^p\, \dom(x)^{\beta-p}\,d\mu
   \leq C\int_{\Omega} g_u(x)^p \dom(x)^{\beta}\,d\mu
\end{equation*}
holds for all $u\in \Lip_0(\Omega)$  
and for all upper gradients $g_u$ of $u$.
Here $\dom(x)=\dist(x,\Omega^c)$ is
the distance from $x\in\Omega$ to 
the complement $\Omega^c=X\setminus\Omega$,  
and in the case $X=\R^n$ we have $g_u=|\nabla u|$.

There is a well-known dichotomy concerning domains admitting a Hardy inequality: 
either the complement of the domain is large (or ``thick'')  
or sufficiently ``thin''. For instance, if 
an open set $\Omega\sub\R^n$ admits a $(p,\beta)$-Hardy inequality, then
there exists $\delta>0$ such that for each ball $B\sub\R^n$ either
$\dimh(2B\cap\Omega^c)>n-p+\beta+\delta$ or
$\dima(B\cap\Omega^c)<n-p+\beta-\delta$; see \cite{KZ} (the case $\beta=0$) and \cite{LMM}. 
Here $\dimh$ denotes the Hausdorff dimension and $\dima$ is the (upper) Assouad dimension 
(see Section \ref{sect:dim} for definitions).

Reflecting this dichotomy, sufficient conditions for the validity of a $(p,\beta)$-Hardy 
inequality can be given in both of the above cases. For thick complements, a canonical 
sufficient condition for the unweighted ($\beta=0$) $p$-Hardy inequality in $\Omega$ is 
the uniform $p$-fatness of $\Omega^c$, or equivalently a uniform Hausdorff content density
condition for $\Omega^c$, see~\cite{Lewis,W,kole}. 
In $\R^n$, uniform $p$-fatness of $\Omega^c$ 
implies in particular that $\dimh(2B\cap\Omega^c)>n-p$ for all balls
centered at $\Omega^c$. On the other hand, in the case of thin complements  
the smallness of the (upper) Assouad dimension of the complement ($\dima(\Omega^c)<n-p$)
is known to be sufficient for the $p$-Hardy inequality; see \cite{KZ,LMM} and note that 
in this case these results are 
based on the works of Aikawa \cite{A,AE}. 
See also \cite{BMS, KLT, kole, LPAMS2, LS} for sufficient conditions 
for weighted Hardy inequalities and to metric space versions of such results.

The main purpose of this paper is to sharpen the previously known sufficient conditions for the 
validity of Hardy inequalities in the case where the complement is assumed 
to be thin. More precisely, we prove the following theorem in the setting of a doubling metric 
space $X$ supporting certain Poincar\'e inequalities (cf.\ Section \ref{sect:dim}). Here
the thinness is formulated in terms of the so-called
lower Assouad codimension of $\Omega^c$
(a metric space version of the (upper) Assouad dimension, see Section~\ref{sect:dim}).

\begin{thm}\label{thm:main}
Let $1 \le p<\infty$ and $\beta<p-1$, and 
assume that $X$ is an unbounded doubling metric space.
If $\beta\le 0$, we further assume that $X$ supports a $p$-Poincar\'e inequality, and
if $\beta > 0$ we assume that $X$ supports a $(p-\beta)$-Poincar\'e inequality.
If $\Omega\sub X$ is an open set satisfying 
\[\lcodima(\Omega^c)>p-\beta,\] 
then $\Omega$ admits a $(p,\beta)$-Hardy inequality.
\end{thm}

In the unweighted case $\beta=0$, which is the most important and most interesting,
Theorem~\ref{thm:main} shows that a $p$-Hardy inequality 
holds in $\Omega$ under the assumptions that $X$ supports a $p$-Poincar\'e inequality 
and $\lcodima(\Omega^c)>p>1$;
in a $Q$-regular space the latter condition is equivalent to $\dima(\Omega^c)<Q-p$.
In particular, this gives a complete answer to a question of 
Koskela and Zhong~\cite[Remark~2.8]{KZ}. See also Corollary~\ref{coro:nonzero bdry}
for an improvement concerning the boundary values of test-functions in
Theorem~\ref{thm:main}.

In $\R^n$, the case $\beta=0$ of Theorem~\ref{thm:main} coincides with
the above-mentioned results from \cite{KZ,LMM}, but our approach gives
a completely new proof is this case. For $\beta\neq 0$ the result is new even in Eulidean spaces. 
Our proof of Theorem~\ref{thm:main} follows the general scheme of Wannebo~\cite{W}:
We first prove $(p,\beta)$-Hardy inequalities for $\beta<0$,
with a suitable control for the constants in the inequalities for $\beta$ close to $0$, and 
then elementary --- but slightly technical --- integration tricks yield the 
inequalities for $0\le\beta<p-1$.

Another goal of this work is to
bring together much of the recent research on Hardy inequalities (see
e.g.\ \cite{KLT,kole,KZ,LMM,lesi,lene,LPAMS2,LS,LT}) in a unified manner in the setting of metric spaces.
For instance, 
it was shown in~\cite{LPAMS2} that 
an open set $\Omega\sub X$ admits a $(p,\beta)$-Hardy inequality 
if the complement $\Omega^c$ satisfies a uniform density condition in terms of 
a Hausdorff content of codimension $q<p-\beta$.   
In the present paper, we establish a new characterization for the upper Assouad codimension 
by means of Hausdorff co-content density, see Corollary~\ref{coro:h char of ucodima}. 
(In Ahlfors regular spaces, such a characterization was observed in~\cite[Remark~3.2]{KLV}.)
Consequently, we obtain the following sufficient 
condition for Hardy inequalities in terms of the upper Assouad codimension, 
which provides a natural counterpart for Theorem~\ref{thm:main}
and shows that there exists a nice ``duality'' between the 
sufficient conditions for the 
cases of thick and thin complements.

\begin{thm}\label{thm:main fat}
Let $1 < p<\infty$ and $\beta<p-1$, and 
assume that $X$ is a doubling metric space  
supporting a $p$-Poincar\'e inequality if $\beta\le 0$, and
a $(p-\beta)$-Poincar\'e inequality if $\beta > 0$.
Let $\Omega\sub X$ be an open set satisfying 
\[\ucodima(\Omega^c)<p-\beta,\]
and, in case $\Omega$ is unbounded, we require in addition that $\Omega^c$
is unbounded as well.
Then $\Omega$ admits a $(p,\beta)$-Hardy inequality. 
\end{thm}

Since the Euclidean space $\R^n$ is $n$-regular and
supports $p$-Poincar\'e inequalities whenever $1\le p<\infty$, 
for $X=\R^n$ the results of Theorems~\ref{thm:main} 
and~\ref{thm:main fat} can be formulated as follows:

\begin{coro}\label{coro:hardy in rn}
 Let $1 < p<\infty$ and $\beta<p-1$, and let $\Omega\sub\R^n$
 be an open set. If
 \[
 \udima(\Omega^c) < n -p + \beta \quad\text{ or }\quad 
 \ldima(\Omega^c) > n -p + \beta,
 \]
 then $\Omega$ admits a $(p,\beta)$-Hardy inequality; 
 in the latter case, if $\Omega$ is unbounded, then we require that also $\Omega^c$ is unbounded.
\end{coro}

Here $\udima=\dima$ is the (upper) Assouad dimension and 
$\ldima$ is the lower Assouad dimension (see Section~\ref{sect:dim}).
In fact, Corollary~\ref{coro:hardy in rn} holds, with $n$ replaced with $Q$,
in any $Q$-regular metric space supporting a $1$-Poincar\'e inequality;
prime examples of such spaces are the Carnot groups.
Let us mention here that in the recent work~\cite{DV}, 
which has been prepared independently of the present paper,
the authors establish similar sufficient conditions for \emph{fractional}
Hardy inequalities in $\R^n$.

The sharpness of Theorems~\ref{thm:main} and~\ref{thm:main fat} will be discussed
in detail in Section~\ref{sect:sharp examples}, but let us mention here some of the
relevant facts. First of all, the bound $p-\beta$ for the codimensions is very
natural and sharp. Indeed, we show in Theorem~\ref{thm:dich for bdry} that if
$\Omega$ admits a $(p,\beta)$-Hardy inequality, then either
$\codimh(\Omega^c) < p-\beta$ or $\lcodima(\Omega^c) > p-\beta$,
and it is clear that for sufficiently regular $\Omega^c$ we have 
$\codimh(\Omega^c)=\ucodima(\Omega^c)$.
Nevertheless, it is not necessary for a $(p,\beta)$-Hardy inequality that either
$\ucodima(\Omega^c) < p-\beta$ or $\lcodima(\Omega^c) > p-\beta$, since 
suitable local combinations of the assumptions in 
Theorems~\ref{thm:main} and~\ref{thm:main fat} 
yield sufficient conditions for Hardy inequalities as well (cf.\ Section~\ref{sect:combo}), 
and in such cases typically only the bound
$\codimh(\Omega^c) < p - \beta$ is satisfied.
There is also a corresponding local dimension dichotomy for Hardy inequalities,
see Theorem~\ref{thm:KZ}.

The requirement $p-\beta>1$ is also sharp in both of the theorems. For instance,
the unit ball $B=B(0,1)\sub\R^n$ gives a simple counterexample for 
the result of Theorem~\ref{thm:main fat} in the case $0< p-\beta \le 1$,
since $\ucodima(\R^n\setminus B) = 0$, but $B$ admits $(p,\beta)$-Hardy inequalities
only when $p-\beta>1$. However, under additional conditions on $\Omega$ 
these Hardy inequalities can also be proven in the case $p-\beta\le 1$, 
see Remark~\ref{rmk:direct} and the discussion in
Section~\ref{sect:sharp examples}.
The unboundedness assumptions of $X$ and $\Omega^c$ in Theorems~\ref{thm:main}
and~\ref{thm:main fat}, respectively, can not be relaxed either,
as the examples at the end of Section~\ref{sect:sharp examples} show. 
The only assumption whose role is not completely understood at the moment is the 
$(p-\beta)$-Poincar\'e inequality in the cases $\beta>0$ of the theorems;
a $p$-Poincar\'e inequality is certainly necessary in any of the cases.
See Remark~\ref{rmk:direct} for a related discussion.

Part of the motivation for the present work stems from the connection between Hardy inequalities 
and the so-called quasiadditivity property of the 
variational capacity. Some aspects of such a connection have been visible e.g.\ in~\cite{A2,AE,KZ,LMM}, 
but only recently it was shown in~\cite{LS} that quasiadditivity of the $p$-capacity with respect 
to $\Omega$ and the validity of a $p$-Hardy inequality in $\Omega$ are essentially equivalent conditions 
(under some mild assumptions on the space $X$ or the open set $\Omega$). 
An assumption equivalent to the
dimension bound $\lcodima(E)>p$ (or rather $\udima(E)<n-p$) was used already by Aikawa~\cite{A}
in connection to the quasiadditivity of the Riesz capacity $R_{1,p}$ with respect to
Whitney decompositions of the complement $\R^n\setminus E$. In order to obtain a corresponding result for
the variational $p$-capacity (in metric spaces),
the unweighted case $\beta=0$ of Theorem~\ref{thm:main}
was deduced in~\cite[Prop.\ 3]{LS} under an additional accessibility condition
(note that the lower Assouad codimension is called the \emph{Aikawa codimension} in~\cite{LS}
and in~\cite{LT}, see Section~\ref{sect:dim} for a discussion). 
Now Theorem~\ref{thm:main} makes such an additional condition unnecessary, and thus we 
have the following corollary to Theorem~\ref{thm:main} and~\cite[Thm 1]{LS},
yielding a complete analogy with the results of Aikawa~\cite{A,AE}. 

\begin{coro}\label{coro:qadd}
 Let $1 < p<\infty$, and 
 assume that $X$ is an unbounded doubling metric space supporting a $p$-Poincar\'e inequality.
 If $\Omega\sub X$ is an open set satisfying $\lcodima(\Omega^c)>p$, 
 then the variational $p$-capacity $\capp{p}(\cdot,\Omega)$ is quasiadditive with respect to Whitney
 covers $\W_c(\Omega)$ for suitably small parameters $c>0$.
\end{coro}

We refer to~\cite{LS} for all the relevant definitions.
Let us also point out that the proof of the corresponding Hardy inequalities in~\cite{LS} 
is more straight-forward than the proof of Theorem~\ref{thm:main} here, and so 
the proof from~\cite{LS} may 
actually be preferred in the cases where the accessibility condition is known to hold.

The organization of the rest of the paper is as follows. 
In Section~\ref{sect:dim} we recall the necessary background material
concerning metric spaces and the various notions of dimension.
Section~\ref{sect:negative beta} contains a proof of the case $\beta<0$
of Theorem~\ref{thm:main}, and the case $0\le\beta<p-1$ is then
established in the following Section~\ref{sect:main}. 
The relation between Hausdorff (co)content
density and the upper Assouad codimension is studied in Section~\ref{sect:fat}
with the help of a measure distribution procedure. 
This section also contains the proof of Theorem~\ref{thm:main fat}.
The necessary conditions for Hardy inequalities are the
topic of Section~\ref{sect:nec}. Finally,
in Section~\ref{sect:combo} we discuss the case where the complement
contains both thick and thin parts, and in 
Section~\ref{sect:sharp examples} we give examples which indicate the
sharpness of our assumptions.

For the notation we remark that $C$ and $c$ 
will denote positive constants whose values are
not necessarily the same at each occurrence. 
If there exist constants $c_1,c_2>0$ such that $c_1\,F\le
G\le c_2F$, we sometimes write $F\simeq G$ and say that $F$ and $G$ are comparable.

\section{Metric spaces and concepts of dimension}\label{sect:dim}

We assume 
throughout this paper that $X=(X,d,\mu)$ is a complete metric measure space, where
$\mu$ is a Borel measure supported on $X$, with $0<\mu(B)<\infty$ whenever 
$B=B(x,r):=\{y\in X : d(x,y)\le r\}$ is a (closed) ball in $X$. 
In addition, we assume 
that $\mu$ is \emph{doubling}, that is, there is a constant $C>0$ such that 
whenever $x\in X$ and $r>0$, we have
\[
  \mu(B(x,2r))\le C\, \mu(B(x,r)).
\]
The completeness of $X$ is actually not needed in all of our results,
but for simplicity we still keep this as a standing assumption.
We also make the tacit assumption that each ball $B\sub X$ has a fixed center $x_B$ 
and radius $\rad(B)$ (but these need not be unique), 
and thus notation such as $\lambda B = B(x_B,\lambda \rad(B))$ 
is well-defined for all $\lambda>0$. The diameter of a set $E\sub X$ is denoted $\diam(E)$,
the distance from a point $x$ to $E$ is $\dist(x,E)$,
and $\Char{E}$ denotes the characteristic function of $E$.

We also say that the measure $\mu$ is $Q$-regular, if there is a constant $C \ge 1$ such that
\[
  C^{-1}r^Q \le \mu(B(x,r)) \le Cr^Q
\]
for all $x \in X$ and every $0 < r < \diam(X)$.

Given
a measurable function $f\colon X\to[-\infty,\infty]$, 
a Borel measurable non-negative function $g$ on $X$
is an \emph{upper gradient} of $f$ if whenever $\gamma$ is a compact rectifiable curve in $X$, we have
\[
  |f(y)-f(x)|\le \, \int_\gamma g\, ds.
\]
Here $x$ and $y$ are the two endpoints of $\gamma$, and the above condition should be interpreted 
as claiming that
$\int_\gamma g\, ds=\infty$ whenever at least one of $|f(x)|, |f(y)|$ is infinite. 
See e.g.~\cite{BB,HEI,HKST} for introduction on analysis on metric spaces
based on the notion of upper gradients.

In addition to the doubling property,
we will also assume throughout the paper that the space $X$ supports a $(1,p)$-Poincar\'e 
inequality (or simply $p$-Poincar\'e inequality) for $1\le p<\infty$,
that is, there exist constants 
$C>0$ and $\lambda\ge 1$ such that whenever $B=B(x,r)\subset X$ and $g$ is an upper gradient 
of a measurable function $f$, we have
\[
    \vint_{B} |f-f_B|\, d\mu \le C\, r\, \left(\,\vint_{\lambda B}g^p\, d\mu\right)^{1/p}
\]
where 
\[
   f_B:=\frac{1}{\mu(B)}\, \int_B f\, d\mu =:\vint_{B}\, f\, d\mu.
\]
To be more precise, we keep a $p$-Poincar\'e inequality as a standing assumption,
but as was already seen in Theorems~\ref{thm:main} and~\ref{thm:main fat},
we occasionally require even stronger Poincar\'e inequalities.

Moreover, we will rely in some of our formulations on the
fundamental result of Keith and Zhong~\cite{KeZ} 
on the self-improvement of Poincar\'e inequalities:
If $1<p<\infty$ and a complete doubling metric space $X$ supports a $p$-Poincar\'e inequality,
then there exist $1\le p_0<p$ such that $X$ supports also
a $p_0$-Poincar\'e inequality, and hence actually $p'$-Poincar\'e inequalities
for all $p'\ge p_0$.

One consequence of Poincar\'e inequalities for the geometry of $X$ is that
a space supporting a $p$-Poincar\'e inequality is \emph{quasiconvex}.
This means that there exists $C\ge 1$ such that
each pair of points $x,y\in X$ can be joined using
a rectifiable curve $\gamma_{x,y}$ of length
$\ell(\gamma_{x,y})\le C d(x,y)$; see e.g.~\cite{BB} for details.
From quasiconvexity we obtain the useful fact  
that if $\Omega\sub X$ is an open set, then
$\dom(x):=\dist(x,\Omega^c)\le C \dist(x,\bdry\Omega)\le C\dom(x)$
for all $x\in\Omega$.

Let $\Omega\sub X$.
A function $u\colon \Omega\to \R$ is said to be \emph{($L$-)Lipschitz}, if
\[
|u(x)-u(y)|\leq L d(x,y)\qquad \text{ for all } x,y\in \Omega.
\]
We denote the set of all Lipschitz functions $u\colon \Omega\to\R$
by $\Lip(\Omega)$. 
In addition, $\Lip_0(\Omega)$ (resp.\ $\Lip_b(\Omega)$) denotes 
the set of Lipschitz functions 
with compact (resp.\ bounded) support in $\Omega$.
Recall that the support of a function $u\colon \Omega\to \R$,
denoted $\spt(u)$, is the closure of the set where $u$ is non-zero.

The \emph{upper and lower pointwise Lipschitz constants} 
of a function $u\colon \Omega\to\R$ at $x\in\Omega$ are
\[
\Lip(u;x)=\limsup_{y\to x} \sup_{y\in B(x,r)} \frac{|u(x)-u(y)|}{r}
\]
and
\[
\lip(u;x)=\liminf_{y\to x} \sup_{y\in B(x,r)} \frac{|u(x)-u(y)|}{r},
\]
respectively.
It is not hard to see that both of these
are upper gradients of a (locally) Lipschitz function
$u\colon \Omega\to \R$ (cf.~\cite[Proposition~1.14]{BB}). 
In $\R^n$, on the other hand, $|\nabla u|$ is a
(minimal weak) upper gradient of $u\in\Lip(\R^n)$, 
consult e.g.\ \cite{BB} for the result and the terminology.

Let us now recall the various notions of dimension that will be important for us throughout the paper.
Let $E\sub X$. The \emph{(upper) Assouad dimension} of $E$, denoted $\udima(E)$ (or simply $\dima(E)$), 
is the infimum of exponents $s\ge 0$ for which 
there is a constant $C \ge 1$ such that for all $x\in E$ and every $0<r<R<\diam(X)$,
the set $E\cap B(x,R)$ can be covered by at most $C(r/R)^{-s}$ balls of radius $r$.
Notice that for $\diam(E)\le r<\diam(X)$ this condition is trivial. 
We remark that this upper Assouad dimension is the ``usual'' Assouad dimension
found in the literature.
See Luukkainen~\cite{luukas} for the basic properties and a historical account on the (upper) Assouad dimension.

Conversely to the above definition, in~\cite{KLV} the \emph{lower Assouad dimension} of $E$, $\ldima(E)$, 
was defined to be the supremum of exponents $t\ge 0$ for which 
there is a constant $c>0$ so that if $0<r<R<\diam(E)$, then for every $x \in E$ at least $c(r/R)^{-t}$ balls of radius $r$ are needed to cover $E\cap B(x,R)$; if $\diam(E)=0$, we omit the upper bound for $R$. 
Closely related concepts have been considered e.g.\ by Larman~\cite{Larman}
and Farser~\cite{Fraser}, but an important difference in our definition
is that we consider all radii $0<r<\diam(E)$, not just small radii; in the 
context of Hardy inequalities this turns out to be essential.

For comparison, recall that the \emph{upper Minkowski dimension} 
of a compact $E\sub X$, denoted
$\udimm(E)$, is the infimum of $\lambda\ge 0$ such that the whole set
$E$ can be covered by at most $C r^{-\lambda}$ balls of radius $0<r<\diam(E)$,
and the \emph{lower Minkowski dimension}, $\ldimm(E)$, 
is the supremum of $\lambda\ge 0$ for which  
at least $c r^{-\lambda}$ balls of radius $0<r<\diam(E)$ are needed to cover $E$.
It follows immediately that 
$\ldima(E) \le \ldimm(E) \le \udimm(E) \le \udima(E)$.

Since a doubling metric space is separable, there exists 
for all $r>0$ a maximal $r$-packing of $E\sub X$, 
that is, a countable collection $\B$ of pairwise disjoint balls 
$B(x_i,r)$, with $x_i\in E$, such that for each $x \in E$ there is $B \in \B$ 
intersecting $B(x,r)$. It is obvious that if $\{ B_i \}_i$ is a 
maximal packing of $E$, then $\{ 2B_i \}_i$ is a cover of $E$.

When working in a (non-regular) metric space $X$, 
it is often convenient to describe the sizes of sets in terms of 
\emph{codimensions} rather than dimensions. For instance, the 
\emph{Hausdorff codimension} of $E\sub X$ 
(with respect to $\mu$) is the number
\[\codimh(E) = \sup\big\{q\geq 0 : \Ha_R^{\mu,q}(E)=0\big\},\]
where 
\[
\Ha_R^{\mu,q}(E) = \inf\bigg\{ \sum_{k} \rad(B_k)^{-q}\mu(B_k) : E \subset \bigcup_{k} B_k,\ 
                               \rad(B_k) \le R \bigg\}
\] 
is the \emph{Hausdorff content of codimension $q$};
if $\mu(E)>0$, then we set $\codimh(E)=0$.
If $\mu$ is $Q$-regular, 
then we have for all $E\sub X$ that $Q\,-\,\codimh(E)=\dimh(E)$, the usual Hausdorff dimension.

We define next the Assouad codimensions following~\cite{KLV}: 

When $E\sub X$ and $r>0$, the (open) $r$-neighborhood of $E$ is the set 
$E_r=\{x\in X:\dist(x,E)<r\}$.
The \emph{lower Assouad codimension}, denoted $\lcodima(E)$, is the supremum of
all $t \ge 0$ for which there exists a constant $C \ge 1$ such that
\begin{equation*}\label{eq:bouli*}
\frac{\mu(E_r\cap B(x,R))}{\mu(B(x,R))}\le C\Bigl(\frac r R\Bigr)^t
\end{equation*}
for every $x\in E$ and all $0<r<R<\diam(X)$.
Conversely, the \emph{upper Assouad codimension} of $E\sub X$,
denoted $\ucodima(E)$,
is the infimum of all $s \ge 0$ for which
there is $c>0$ such that
\begin{equation*}\label{eq:bouliconv*}
\frac{\mu(E_r \cap B(x,R))}{\mu(B(x,R))} \geq c\Bigl(\frac r R\Bigr)^s
\end{equation*}
for every $x\in E$ and all $0<r<R<\diam(E)$. If $\diam(E)=0$, we omit the upper bound for $R$.

If $\mu$ is $Q$-regular, then it is not hard to see that
\begin{equation*}\label{eq:dimai ja codima*}
\udima(E)  = Q - \lcodima(E)\quad\text{ and }\quad
\ldima(E)  = Q - \ucodima(E)
\end{equation*}
for all $E\sub X$ (cf.~\cite{KLV}).

\begin{remark}\label{rmk:LT}
It was shown in~\cite[Thm.~5.1]{LT} that the
lower Assouad codimension can also be characterized as
the supremum of all $q\geq 0$ for which there exists a constant $C \ge 1$ such that
\begin{equation}\label{eq:aikawa}
  \int_{B(x,r)} \dist(y,E)^{-q}\,d\mu(y) \le C r^{-q}\mu(B(x,r))
\end{equation}
for every $x\in E$ and all $0<r<\diam(X)$.
(Here we interpret the integral to be $+\infty$ if $q>0$ and $E$ has positive measure.)
\end{remark}

A concept of dimension defined via integrals as in~\eqref{eq:aikawa}
was used by Aikawa in~\cite{A} for subsets of $\R^n$
(see also~\cite{AE}).
Thus, in~\cite{LS,LT}, where the interest originates from such
integral estimates, the lower Assouad codimension was called
the \emph{Aikawa codimension}. We will see later, especially in Section~\ref{sect:nec},
that this kind of integral estimates arise very naturally in 
connection to Hardy inequalities.

The next lemma records the fact that
the Aikawa condition~\eqref{eq:aikawa} enjoys
self-improvement. This property is
a direct consequence of the famous 
self-improvement result for reverse
H\"older inequalities (in $\R^n$ due to
Gehring~\cite{geh}), and we will need this in Section~\ref{sect:nec}
when proving our necessary conditions 
for Hardy inequalities.
For $q>1$, the result of
Lemma~\ref{lemma:aikawa si} is
contained in the proof of Lemma~2.4 in~\cite{KZ},
and in the proof of Proposition~4.3 in~\cite{IV2} the
same fact is used in $\R^n$.

\begin{lemma}\label{lemma:aikawa si}
 Let $0<R_0\le \infty$ and assume that $E\sub X$ 
 satisfies the Aikawa condition~\eqref{eq:aikawa}
 for every $x\in E$ and all $0<r<R_0$ with an exponent $q>0$
 and a constant $C_0>0$. 
 Then there exist $\delta>0$ and $C>0$, depending only on the 
 given data, such that condition~\eqref{eq:aikawa} holds
 for every $x\in E$ and all $0<r<R_0$ with the exponent $q+\delta$
 and the constant $C$. 
\end{lemma}

\begin{proof}
The proof is based on the metric space version of the Gehring Lemma;
see e.g.~\cite[Thm.~3.22]{BB} or~\cite[p.~11]{ST}.

Fix any $0<s<q$, e.g.\ $s=q/2$, and let $B_0=B(x,R)$
with $x\in E$ and $0<R<R_0$.
In addition, let $B$ be a ball such that $B\sub B_0$.
If $4B\cap E\neq \emptyset$, then 
we find a ball $B'$ centered at $E$ and
with a radius comparable to $\rad(B)$ such that $B\subset B'$,
and thus we have by~\eqref{eq:aikawa} and doubling that
\begin{equation*}\label{eq:gehr1*}
\begin{split}
 \int_{B} \dist(y,E)^{-q}\,d\mu 
 & \le \int_{B'} \dist(y,E)^{-q}\,d\mu \le C_0 \mu(B')\rad(B')^{-q} \\
 &\le C \mu(B) \bigl(\rad(B)^{-s}\bigr)^{q/s} 
  \le C \mu(B) \biggl(\vint_{2B} \dist(y,E)^{-s}\,d\mu\biggr)^{q/s}\,;
\end{split}
\end{equation*}
in the last inequality we used the fact that
$\dist(y,E)^{-s}\ge C \rad(B)^{-s}$
for all $y\in 2B$.
In particular we obtain the reverse H\"older inequality
\begin{equation}\label{eq:gehr1}
 \biggl(\vint_{B} \dist(y,E)^{-q}\,d\mu\biggr)^{s/q} 
 \le C \vint_{2B} \dist(y,E)^{-s}\,d\mu.
\end{equation}
On the other hand, if $4B\cap E = \emptyset$, then
$\dist(y,E)\simeq \dist(x,E)$ for all $y\in 2B$, and 
thus~\eqref{eq:gehr1} holds in this case as well.

Since the function $f(y)= \dist(y,E)^{-s}\in L^1(B_0)$ now
satisfies the assumption of the Gehring Lemma~\cite[Thm.~3.22]{BB}
for all balls $B\sub B_0$, the proof in~\cite{BB}
shows that there is $\delta>0$ such that
we have for the ball $B_1 =\frac 1 2 B_0$ that
\begin{equation*}\label{eq:gehr2}
\begin{split}
 \biggl(\vint_{B_1} \dist(y,E)^{-(q+\delta)}\,d\mu\biggr)^{s/(q+\delta)} 
 & \le C \vint_{2B_1} \dist(y,E)^{-s}\,d\mu \\ 
 & \le C \biggl(\vint_{2B_1} \dist(y,E)^{-q}\,d\mu\biggr)^{s/q} \le C \rad(B_1)^{s}, 
\end{split}
\end{equation*}
where we also used H\"older's inequality and the original estimate~\eqref{eq:aikawa}.
Moreover, here the constant $C>0$ is independent of the ball $B_1$. 
The claim follows, since for balls $B_1$ with $R_0/2\le \rad(B_1)<R_0$ we
can consider a cover using smaller balls.
\end{proof}

\section{A weighted Hardy inequality for $\beta<0$}\label{sect:negative beta}

As a first step towards Theorem~\ref{thm:main},
we establish in this section the result in the case $\beta<0$.
The particular form of the constant in the $(p,\beta)$-Hardy inequalities below
plays an important role in the proof of the general case of Theorem~\ref{thm:main}.
Notice that here the test functions are not required to vanish in $\Omega^c$,
and recall that we have the standing assumption that
$X$ is a complete doubling metric space supporting a $p$-Poincar\'e inequality.

\begin{proposition}\label{prop:hardy and assouad beta<0}
Assume that $X$ is unbounded and
let $1\le p<\infty$ and $\beta<0$.
If $\Omega\sub X$ is an open set with $\lcodima(\Omega^c)>p-\beta$,  
then $\Omega$ admits a $(p,\beta)$-Hardy inequality, and, in fact, the 
$(p,\beta)$-Hardy inequality holds for all $u\in\Lip_b(\Omega)$.

Moreover, there exists $\wtilde\beta<0$ 
such that for $\tilde\beta<\beta<0$ the constant in the $(p,\beta)$-Hardy inequality can be chosen to be
$C=|\beta|^{-1}C^*>0$, where $C^*>0$ is independent of $\beta$. 
\end{proposition}

\begin{proof}
Write $E=\Omega^c$.
For each $k\in\Z$, let $\B_k=\{B_{k,i}\}$ be a maximal packing of $E$ with balls
$B_{k,i}=B(x_{k,i},2^k)$, $x_{k,i}\in E$, and write 
$N_k=\bigcup_i 4B_{k,i}$ and $A_k=N_{k}\setminus N_{k-1}$.
We can then 
choose for each $B=B_{k,i}\in\B_k$ balls $B^j\in\B_j$, $j\geq k$, such that
$B=B^k$ and
$4B^j\sub 4B^{j+1}$ for all $j\geq k$. Indeed,
if $x^j$ is the center of $B^j$, there is $B^{j+1}=B(x^{j+1},2^{j+1})\in\B_{j+1}$
such that $x^j\in 2 B^{j+1}$, and hence
$4B^j\subset B(x^{j+1},2\cdot 2^{j+1}+4\cdot 2^{j})=4B^{j+1}$.
In particular, it follows that $B\subset 4B^j$ for all $j\ge k$.

Let $u\in\Lip_b(\Omega)$. 
Since $X$ is unbounded, we have for each $B_{k,i}$ that
$\vint_{4B_{k,i}^j}u \rightarrow 0$ as ${j\to\infty}$.
A standard telescoping trick using the $p$-Poincar\'e inequality
then yields for every $B=B_{k,i}\in\B_k$ that
\begin{equation}\label{eq: chain}
 |u_{4B}|\leq \sum_{j=k}^{\infty}|u_{4B^j}-u_{4B^{j+1}}| 
 \leq C \sum_{j=k}^{\infty} 2^{j} \biggl(\vint_{4\lambda B^j}g_u^p\,d\mu\biggr)^{1/p}.
\end{equation}
Comparison of the sum on the right-hand side of \eqref{eq: chain} with the convergent 
geometric series $\sum_{j=k}^\infty 2^{(k-j)\delta}$, for any $\delta>0$,
shows that there exists a constant $C_1(\delta)>0$, independent of $u$ and $B$, 
and an index $j(B)\geq k$
such that
\begin{equation}\label{eq: one big}
2^{j(B)} \biggl( \vint_{4\lambda B^{j(B)}}g_u^p\,d\mu\biggr)^{1/p} \geq 
   C_1 |u_{4B}|2^{(k-j(B))\delta}.  
\end{equation}
Now fix $q_1,q_2$ such that $\lcodima(\Omega^c)>q_1>q_2>p-\beta$, and
set $\delta=(q_2-p+\beta)/p>0$. We obtain from 
\eqref{eq: one big} for each $B\in\B_k$ a ball $B^{j(B)}$ 
of radius $2^{j(B)}$ satisfying 
\begin{equation}\label{eq:the one ball}
2^{k(q_2-p+\beta)}|u_{4B}|^p
\leq C_1 (2^{j(B)})^{q_2+\beta} \mu \big(B^{j(B)}\big)^{-1} \int_{4\lambda B^{j(B)}}g_u^p\,d\mu.
\end{equation}

Let us now start to estimate the left-hand side of the $(p,\beta)$-Hardy inequality.
Since $\dom(x)\ge 2^{k-1}$ for $x\in A_k$, we have
\begin{equation}\label{eq:compute}\begin{split}
 \int_\Omega & |u|^p  \dom^{\beta-p}\,d\mu  
   = \sum_{k=-\infty}^\infty \int_{A_k} |u|^p\dom^{\beta-p}\,d\mu\\
 &\leq 2^{p-\beta} \sum_{k=-\infty}^\infty 2^{k(\beta-p)} \int_{A_k} |u|^p\,d\mu 
  \leq 2^{p-\beta} \sum_{k=-\infty}^\infty 2^{k(\beta-p)} \sum_{B\in\B_k}\int_{4B} |u|^p\,d\mu\\
 &\leq C_2 \sum_{k=-\infty}^\infty 2^{k(\beta-p)} \sum_{B\in\B_k}\int_{4B} |u-u_{4B}|^p\,d\mu 
     + C_2 \sum_{k=-\infty}^\infty 2^{k(\beta-p)} \sum_{B\in\B_k}\int_{4B} |u_{4B}|^p\,d\mu,
\end{split}
\end{equation}
where $C_2=2^{p-\beta} C'$ and $C'=2^{p-1}>0$ is independent of $\beta$.
The first sum in the last line of~\eqref{eq:compute} can be estimated 
with the help of the $(p,p)$-Poincar\'e inequality 
(which is a well-known consequence of the $p$-Poincar\'e inequality, see e.g.~\cite{HaK,HEI})
and the controlled overlap of the balls $4\lambda B$ for $B\in\B_k$ (with a fixed $k\in\Z$, 
by doubling).
We also rewrite the integral, change the order of summation, 
and use the simple estimate $\dom(x)\le 4\cdot 2^{i}$ for $x\in A_i$, as follows:
\begin{equation}\label{eq:first sum}\begin{split}
 \sum_{k=-\infty}^\infty 2^{k(\beta-p)} & \sum_{B\in\B_k}\int_{4B} |u-u_{4B}|^p\,d\mu 
 \le C \sum_{k=-\infty}^\infty 2^{k\beta} \sum_{B\in\B_k} \int_{4\lambda B}g_u^p\,d\mu\\
 & \le C \sum_{k=-\infty}^\infty 2^{k\beta} \int_{N_{k+M}}g_u^p\,d\mu
  = C 2^{-M\beta} \sum_{k=-\infty}^\infty 2^{k\beta} \int_{N_{k}}g_u^p\,d\mu \\
 &  = C 2^{-M\beta} \sum_{k=-\infty}^\infty 2^{k\beta} \sum_{i=-\infty}^{k}\int_{A_i}g_u^p\,d\mu   
   = C 2^{-M\beta} \sum_{i=-\infty}^\infty \int_{A_i} g_u^p\,d\mu \sum_{k=i}^\infty 2^{k\beta} \\
 &  = C 2^{-M\beta} \sum_{i=-\infty}^\infty \frac {2^{i\beta}}{1-2^\beta} \int_{A_i} g_u^p\,d\mu   
  \le  C \frac {2^{-M\beta}4^{-\beta}}{1-2^\beta} \sum_{i=-\infty}^\infty \int_{A_i} g_u^p \dom^\beta\,d\mu \\ &   \le C \frac {2^{-(M+2)\beta}}{1-2^\beta} \int_{\Omega} g_u^p \dom^\beta\,d\mu.
\end{split}
\end{equation}
Above the constants $C>0$ and $M=M(\lambda)\ge 4$ are independent of $\beta$;
note also how the assumption $\beta<0$ was needed.

In the last sum of~\eqref{eq:compute} we first use~\eqref{eq:the one ball} and 
then change the order of summation:
\begin{equation}\label{eq:2nd sum}\begin{split}
 \sum_{k=-\infty}^\infty & 2^{k(\beta-p)} \sum_{B\in\B_k}\int_{4B} |u_{4B}|^p\,d\mu 
  = \sum_{k=-\infty}^\infty 2^{k(\beta-p)} \sum_{B\in\B_k}\mu(4B) |u_{4B}|^p\\
 & \leq  C \sum_{k=-\infty}^\infty 2^{-q_2k} \sum_{B\in\B_k} \mu(4B) 
   (2^{j(B)})^{q_2+\beta} \mu\big(B^{j(B)}\big)^{-1} \int_{4\lambda B^{j(B)}}g_u^p\,d\mu\\ 
 & = C \sum_{j=-\infty}^\infty \sum_{\tilde B\in\B_j} 2^{j\beta} \int_{4\lambda \tilde B}g_u^p\,d\mu  \sum_{k\leq j}\sum_{\{B\in\B_k:\tilde B=B^{j(B)}\}} 2^{-q_2k} 2^{q_2 j} \mu(4B) \mu(\tilde B)^{-1}\\
 & \leq  C \sum_{j=-\infty}^\infty \sum_{\tilde B\in\B_j} 2^{j\beta} \int_{4\lambda \tilde B}g_u^p\,d\mu  \sum_{k\leq j}\sum_{\{B\in\B_k:B\sub 4\tilde B\}}\frac{\mu(B)2^{-q_2k}}{\mu(\tilde B)2^{-q_2j}}.
 \end{split}
\end{equation}
Since the balls $B$, for $B\in\B_k$, are pairwise disjoint,
the assumption $\lcodima(E)>q_1>q_2$ implies
(recall here that $E_{2^k}=\{x\in X : \dist(x,E) < 2^k\}$)
\begin{equation*}\label{eq:assouad helps*}
\begin{split}
 \sum_{k\leq j}\sum_{\{B\in\B_k:B\sub 4\tilde B\}}\frac{\mu(B)2^{-q_2k}}{\mu(\tilde B)2^{-q_2j}}
 &\leq C \sum_{k\leq j} \frac{\mu(E_{2^{k}}\cap 4\tilde B)}{\mu(4\tilde B)}
   \left(\frac{2^{k}}{2^{j}}\right)^{-q_1}\left(\frac{2^{k}}{2^{j}}\right)^{q_1-q_2}\\
 &\leq C \sum_{k\leq j} \left(\frac{2^{k}}{2^{j}}\right)^{q_1-q_2}\leq C(q_1,q_2).
\end{split}
\end{equation*} 
Thus we obtain from~\eqref{eq:2nd sum}, using also the bounded overlap of $4\lambda\tilde B$, 
that
\begin{equation}\label{eq:2nd sum final} 
\begin{split}
\sum_{k=-\infty}^\infty & 2^{k(\beta-p)} \sum_{B\in\B_k}\int_{4B} |u_{4B}|^p\,d\mu 
\le C \sum_{j=-\infty}^\infty \sum_{\tilde B\in\B_j} 2^{j\beta} \int_{4\lambda \tilde B}g_u^p\,d\mu\\
 &  \leq C \sum_{j=-\infty}^\infty 2^{j\beta} \int_{N_{j+M}}g_u^p\,d\mu
  \le C  \frac {2^{-(M+2)\beta}}{1-2^\beta} \int_{\Omega} g_u^p \dom^\beta\,d\mu,   
\end{split}
\end{equation}
where the last inequality follows just like in~\eqref{eq:first sum}.
A combination of \eqref{eq:compute}, \eqref{eq:first sum}, 
and~\eqref{eq:2nd sum final} thus yields the $(p,\beta)$-Hardy inequality
\begin{equation}\label{eq:look the constant}
 \int_\Omega |u(x)|^p  \dom(x)^{\beta-p}\,d\mu  
 \le C \frac {2^{-(M+3)\beta}}{1-2^\beta} \int_{\Omega} g_u(x)^p \dom(x)^\beta\,d\mu,
\end{equation}
where the constant $C>0$ is independent of $\beta$,
but depends on $\delta$ (cf.~\eqref{eq: chain}), $q_1$, $q_2$, $p$, and the data associated to~$X$.

We conclude the proof with a closer examination of the constant in~\eqref{eq:look the constant}. 
First of all, if $-1<\beta<0$, then $2^{-(M+3)\beta}\le 2^{M+3}$, 
and when $\beta$ is close enough to $0$, then
$1-2^{\beta}\simeq -\beta$. In addition, the constant $C$ in~\eqref{eq:look the constant} 
depends on $\delta$, $q_1$ and $q_2$,
and hence indirectly on $\beta$ as well, since 
$\delta=(q_2-p+\beta)/p$ and
$p-\beta<q_1<q_2<\lcodima(\Omega^c)$. Nevertheless, if e.g.\ 
$\big(p-\lcodima(\Omega^c)\big)/2<\beta<0$, then this constant can 
obviously be chosen to depend only on
$p$ and $\lcodima(\Omega^c)$ (and the data associated to~$X$). 
It follows that there exists $\wtilde\beta<0$,
depending on $p$ and $\lcodima(\Omega^c)$,
such that for 
all $\wtilde\beta<\beta<0$ we have
\[
 \int_\Omega |u(x)|^p  \dom(x)^{\beta-p}\,d\mu  
  \le \frac {C^*}{|\beta|} \int_{\Omega} g_u(x)^p \dom(x)^\beta\,d\mu,
\]
where the constant $C^*>0$ is independent of the particular $\beta$.   
\end{proof}

\section{The case $0\leq\beta<p-1$ of Theorem~\ref{thm:main}}\label{sect:main}

We now turn to the proof of the weighted $(p,\beta)$-Hardy inequality 
in the case $0\le\beta <p-1$
under the assumption $\lcodima(\Omega^c)>p-\beta$. 
The proof is based on Proposition~\ref{prop:hardy and assouad beta<0}, and it combines ideas 
from~\cite{W,KZ,lesi}. In fact, in the Euclidean case the result can be readily 
deduced from the $(p,\beta)$-Hardy inequalities of Proposition~\ref{prop:hardy and assouad beta<0}
with a careful use of~\cite[Lemma 2.1]{lesi}.

\begin{proof}[Proof of Theorem \ref{thm:main}]
For $\beta < 0$, the claim follows from Proposition~\ref{prop:hardy and assouad beta<0}. 
and thus we are left with the case $0\le\beta<p-1$. 
Since $\lcodima(\Omega^c)> p-\beta > 1$, we have by Proposition~\ref{prop:hardy and assouad beta<0} 
that $\Omega$ admits an $(p-\beta,-\beta_0)$-Hardy inequality whenever 
$0<\beta_0<\lcodima(\Omega^c) - p + \beta$; here we need to know that
$X$ supports a $(p-\beta)$-Poincar\'e inequality. Moreover, there exists $\wtilde\beta_0>0$
such that for $0<\beta_0<\wtilde\beta_0$ the constant in the $(p-\beta,-\beta_0)$-Hardy inequality
is $C^*\beta_0^{-1}$, with $C^*>0$ independent of $\beta_0$.
Fix such $\beta_0$ to be chosen later.

Let $u\in \Lip_0(\Omega)$ with an upper gradient $g_u$, 
and define 
\[
v(x)=|u(x)|^{{p}/{(p-\beta)}}\dom(x)^{\beta_0/(p-\beta)}.
\] 
Then $v$ is a Lipschitz-function with a compact support in $\Omega$, 
$|u(x)|^{p} = |v(x)|^{p-\beta}\dom(x)^{-\beta_0}$, and, moreover,
the function
\begin{equation}\label{eq:ugv}
g_v(x) := \tfrac{p}{p-\beta}|u(x)|^{\beta/(p-\beta)}g_u(x) \dom(x)^{\beta_0/(p-\beta)}
        + \tfrac{\beta_0}{p-\beta} |u(x)|^{{p}/{(p-\beta)}} \dom(x)^{(\beta_0-p+\beta)/(p-\beta)}
\end{equation}
is an upper gradient of $v$ (cf.\ e.g.\ \cite[Thm.~2.15 and 2.16]{BB}); 
here it is essential that the support of $u$ is a compact set inside $\Omega$. 
Using the $(p-\beta,-\beta_0)$-Hardy inequality of Proposition~\ref{prop:hardy and assouad beta<0} for
$v$, we obtain
\begin{equation}\label{eq: hardy for v}\begin{split}
 \int_\Omega  |u|^{p}  \dom^{\beta-p}\,d\mu & 
   =  \int_\Omega |v|^{p-\beta} \dom^{-\beta_0-(p-\beta)}\,d\mu\\
  &  \le C^*\beta_0^{-1} \int_\Omega g_v^{p-\beta}\dom^{-\beta_0}\,d\mu.
\end{split}
\end{equation}
By~\eqref{eq:ugv} and H\"older's inequality (for exponents $\frac p\beta$ and $\frac{p}{p-\beta}$),
we estimate the above integral for $g_v$ as
\begin{equation}\label{eq: holder for v}\begin{split}
\int_\Omega g_v^{p-\beta}\dom&^{-\beta_0}\,d\mu 
    \le 2^{p-\beta} \Big(\tfrac{p}{p-\beta}\Big)^{p-\beta}
  \int_\Omega |u|^\beta g_u^{p-\beta}\dom^{\beta_0 - \beta_0}\,d\mu \\
  & \qquad\qquad + 2^{p-\beta} \Big(\tfrac{\beta_0}{p-\beta}\Big)^{p-\beta}
     \int_\Omega |u|^{p}{\dom^{\beta_0 - p + \beta-\beta_0}}\,d\mu\\
  &  \le C(p,\beta) \int_\Omega \Big(|u|^\beta\dom^{\frac{\beta(\beta-p)}{p}}\Big)
       \Big(g_u^{p-\beta}\dom^{\frac{\beta(p-\beta)}{p}}  \Big)\,d\mu\\
  & \qquad\qquad + C(p,\beta) {\beta_0}^{p-\beta}\int_\Omega |u|^{p}{\dom^{\beta - p}}\,d\mu\\  
  & \le C(p,\beta)  \bigg(\int_\Omega |u|^p\dom^{\beta-p}\,d\mu \bigg)^\frac{\beta}{p}
     \bigg(\int_\Omega g_u^p\dom^{\beta}\,d\mu\bigg)^{\frac{p-\beta}{p}}\\
  & \qquad\qquad + C(p,\beta) {\beta_0}^{p-\beta}\int_\Omega |u|^{p}{\dom^{\beta - p}}\,d\mu,  
\end{split}
\end{equation}
where the constant $C(p,\beta) = 2^{p-\beta} \big({p}/{(p-\beta)}\big)^{p-\beta}$ is independent of $\beta_0$.

We now choose $0<\beta_0<\wtilde\beta_0$ to be so small that 
\[
 C^* \beta_0^{-1} C(p,\beta) {\beta_0}^{p-\beta}
 = C^* C(p,\beta) {\beta_0}^{p-\beta-1} < \tfrac 1 2.
\]
This is possible since $p-\beta>1$ and the factor $C^* C(p,\beta)$ does not depend on $\beta_0$.
After the insertion of~\eqref{eq: holder for v} into~\eqref{eq: hardy for v}, 
we observe that under the above choice of $\beta_0$,  
the second term emerging on the right-hand side is less than half of the left-hand side, 
and thus we obtain
\begin{equation}\label{eq: hardy for u}
 \int_\Omega  |u|^{p}  \dom^{\beta-p}\,d\mu  
 \le C \bigg(\int_\Omega |u|^p\dom^{\beta-p}\,d\mu \bigg)^\frac{\beta}{p}
     \bigg(\int_\Omega g_u^p\dom^{\beta}\,d\mu\bigg)^{\frac{p-\beta}{p}}.
\end{equation}
The $(p,\beta)$-Hardy inequality for $u$ now follows from~\eqref{eq: hardy for u}
by dividing with the first factor on the right-hand side (which we may assume to be
non-zero), and then taking both sides to power $p/(p-\beta)$. 
\end{proof}

Notice that the requirement $p-\beta>1$ is essential in the above proof.
This is not merely a technical assumption, since Theorem~\ref{thm:main} need
not hold if $p-\beta \le 1$; see Section~\ref{sect:sharp examples}.

\begin{remark}\label{rmk:direct}
 It would be interesting to know if there is a more direct proof for the case $\beta\ge 0$,
 i.e.\ one avoiding the use of the case $\beta<0$. This is strongly related to the question
 what Poincar\'e inequalities are actually needed in Theorem~\ref{thm:main}. The same question applies
 to Theorem~\ref{thm:main fat} as well in the case $\beta>0$ 
 (cf.\ the proof at the end of Section~\ref{sect:fat}).
 
 Here it is good to recall that when $\Omega^c$ (or actually $\bdry\Omega$)
 satisfies additional accessibility conditions from
 within $\Omega$, then such direct proofs exist. 
 Moreover, under these accessibility conditions the results of
 Theorems~\ref{thm:main} and~\ref{thm:main fat} can be extended to the case
 $\beta\ge p-1$, see e.g.~\cite[Thm.~1.4]{kole}, \cite[Thm.~4.3]{LMM}, 
 and~\cite[Thm.~4.5]{LPAMS2}. 
\end{remark}

\begin{remark}\label{rmk:one point}
An interesting special case of Theorem~\ref{thm:main} is that where 
$X$ is unbounded and
the distance function $\dom(x)$ is replaced by 
the distance to a fixed point $x_0\in X$, i.e., 
we have the inequality
\begin{equation}\label{eq:hardy-point}
\int_{X} |u(x)|^p\, d(x,x_0)^{\beta-p}\,d\mu
   \leq C\int_{X} g_u(x)^p\, d(x,x_0)^{\beta}\,d\mu.
\end{equation}
It follows from Theorem~\ref{thm:main} that this inequality holds
for all $u\in\Lip_0(X\setminus\{x_0\})$ when
$\lcodima(\{x_0\})>p-\beta>1$, and in fact, by 
Corollary~\ref{coro:nonzero bdry} below, $u$ need not vanish at $x_0$,
so the inequality actually holds for all $u\in\Lip_b(X)$.
On the other hand, Theorem~\ref{thm:main fat} implies that
for $\ucodima(\{x_0\})<p-\beta$ inequality~\eqref{eq:hardy-point}
is valid for all $u\in\Lip_0(X\setminus\{x_0\})$.

Let us mention here that 
the lower and upper Assouad codimensions of a point
are closely related to the \emph{exponent sets} of the point $x_0$, 
defined in~\cite{BBL}.
Namely,
$\lcodima(\{x_0\})=\sup \lQ(x_0)$ and $\ucodima(\{x_0\})=\inf\uQ(x_0)$
(see~\cite{BBL} for the definitions of the $Q$-sets).

In the Heisenberg group $\mathbb H_n$, 
which is one particular example of a metric space
satisfying our general assumptions,
an inequality of the type~\eqref{eq:hardy-point} 
was recently obtained by Yang~\cite{yang} using a completely different approach. 
Since $\lcodima(\{0\})=Q:=2n+2$ for $0\in\mathbb H_n$,
inequality~(3.6) in~\cite{yang}
corresponds exactly to inequality~\eqref{eq:hardy-point},
for $x_0=0$, under the condition $1<p<\lcodima(\{x_0\})$\,;
notice that Theorem~1.1 in~\cite{yang} only records the unweighted case $\beta=0$.
However, the requirement $p-\beta>1$ is not needed in~\cite{yang}, 
and the inequality is established even with the sharp constant $\big(p/(Q-p+\beta)\big)^p$.
With our techniques there is no hope of obtaining any sharpness for the constants.

Recall also that
in Euclidean spaces the corresponding well-known inequality, i.e.\ 
\begin{equation*}\label{eq:hardy-point rn}
\int_{\R^n} |u(x)|^p\, |x|^{\beta-p}\,d\mu
   \leq C\int_{\R^n} |\nabla u(x)|^p\, |x|^{\beta}\,d\mu,
\end{equation*}
with the optimal constant $C=\big(p/|n-p+\beta|\big)^p$,
follows easily by using the classical $1$-dimensional weighted
Hardy inequalities (cf.~\cite{HLP}) on rays starting from the origin.
For $p-\beta<n$ this inequality holds for all $u\in\Lip_b(\R^n)$, and
for $p-\beta>n$ for all $u\in\Lip_0(\R^n\setminus\{0\})$.
See also~\cite{SSW} for related inequalities where the distance is taken
to a $k$-dimensional subspace of $\R^n$, $1\le k < n$.
\end{remark}

\section{Upper Assouad codimension and thickness}\label{sect:fat}

In this section we establish a connection between the upper Assouad codimension
and Hausdorff content density conditions, which might also be of independent
interest, and as a consequence obtain a proof for Theorem~\ref{thm:main fat}. 
The following lemma is a modification of \cite[Lemma~4.1]{lene},
where a corresponding statement was given in terms of Minkowski contents in
Euclidean spaces.

\begin{lemma}\label{lemma:a to h}
Let $E\sub X$ be a closed set and assume that 
$\ucodima(E)<q$.
Then 
there exists a constant
$C>0$ such that
\begin{equation}\label{eq: hausd}
\Ha^{\mu,q}_{R}\big(E \cap B(w,R)\big)\ge  C\,R^{-q}\mu(B(w,R))
\end{equation}
for every $w\in E$ and all $0<R<\diam(E)$.
\end{lemma}

\begin{proof}
Let $\ucodima(E)<q'<q$ and fix $0<\delta<1/2$ to be chosen a bit later.
Let also $w\in E$ and $0<R<\diam(E)$, and denote $B_0=B(w,R)$ and $r_k=\delta^k R$.

We begin with a maximal packing $\{B_{i_1}\}_{i_1}$ of $\frac 12 B_0\cap E$
with balls $B_{i_1}=B(w_{i_1},r_1)$, $i_1\in I_0\sub\N$, 
where $w_{i_1}\in \frac 12 B_0\cap E$.
Then we have for the $r_1$-neighborhood of $E$ that 
$E_{r_1}\cap \frac 1 4 B_0 \sub \bigcup_{i_1} 3B_{i_1}$, and
thus doubling and the fact $q'>\ucodima(E)$ imply
\[
\mu(B_0)\Bigl(\frac{r_1}{R}\Bigr)^{q'} 
\le C \mu\bigl(E_{r_1}\cap \tfrac 1 4 B_0\bigr)
\le C \sum_{i_1}\mu(3B_{i_1})\le C \sum_{i_1}\mu(B_{i_1}).
\]
In particular, there exists a constant $c_0>0$, independent of $w$ and $R$, such that
\begin{equation*}
 \sum_{i_1}\mu(B_{i_1}) \geq c_0 \delta^{q'}\mu(B_0).
\end{equation*}
We now choose $0<\delta<1$ to be so small that $\delta^{q-q'}<c_0$,
whence $c_0 \delta^{q'}>\delta^{q}$, and so
\[
M_0:=\sum_{i_1}\mu(B_{i_1}) > \delta^{q}\mu(B_0).
\]
We complete the first step of the construction by defining a measure distribution for the
balls $B_{i_1}$ by 
\begin{equation}\label{eq:nu one}
\nu(B_{i_1})=\mu(B_{i_1})/M_0< \mu(B_{i_1}) \delta^{-q}\mu(B_0)^{-1}.
\end{equation}

In the next step, we create a similar measure distribution inside the
balls $B_{i_1}$.
As above, we find for each $i_1\in I_0$ 
pairwise disjoint balls $B_{i_1 i_2}=B(w_{i_1 i_2},r_2)$, $i_2\in I_{i_1}\sub\N$,
where $w_{i_1 i_2}\in \frac 1 2 B_{i_1}\cap E$ and
\begin{equation}\label{eq:N one}
 M_{i_1}:=\sum_{i_2}\mu(B_{i_1 i_2}) \geq c_0 \delta^{q'}\mu(B_{i_1})>\delta^{q}\mu(B_{i_1}).
\end{equation}
We define 
\[
\nu(B_{i_1 i_2}):=\nu(B_{i_1}) \mu(B_{i_1 i_2})/M_{i_1}  < \mu(B_{i_1 i_2})\delta^{-2q}\mu(B_{0})^{-1},
\]
where the inequality follows from~\eqref{eq:nu one} and~\eqref{eq:N one}. 
Notice that since ${B_{i_1}\cap B_{j_1}=\emp}$ whenever $i_1\neq j_1$\,,
and clearly $B_{i_1 i_2}\sub B_{i_1}$ for every 
$i_2\in I_{i_1}$, we have that all the balls
$B_{i_1 i_2}$ are pairwise disjoint.

Continuing the construction in the same way, 
we find in the $k$:th step 
a collection of pairwise disjoint closed balls 
$B_{i_1 i_2 \dots i_{k-1} i_{k}}\sub B_{i_1 i_2 \dots i_{k-1}}$, 
$i_{k}\in I_{i_1 i_2 \dots i_{k-1}}\sub\N$, 
with center points
$w_{i_1 i_2 \dots i_{k}}\in E\cap \frac 1 2 B_{i_1 i_2 \dots i_{k-1}}$
and all of radius $r_k= \delta^{k} R$,
such that
\[
E_{r_k}\cap \tfrac 1 4 B_{i_1 i_2 \dots i_{k-1}}\sub \bigcup_{i_k} 3 B_{i_1 i_2 \dots i_{k-1} i_{k}}.
\]
Thus $q'>\ucodima(E)$ and the choice of $\delta$ imply
\begin{equation}\label{eq:N kei}
M_{i_1 i_2 \dots i_{k-1}} :=  
\sum_{i_k}\mu(B_{i_1 i_2 \dots i_{k-1} i_{k}}) \ge c_0 \delta^{q'}\mu(B_{i_1 i_2 \dots i_{k-1}})
> \delta^{q}\mu(B_{i_1 i_2 \dots i_{k-1}}).
\end{equation}
We now distribute the measure for the balls $B_{i_1 i_2 \dots i_{k-1} i_{k}}$
as follows:
\begin{equation}\label{eq:nupallo}
\begin{split}
\nu(B_{i_1 i_2 \dots i_{k-1} i_{k}}) := &\  
\nu(B_{i_1 i_2 \dots i_{k-1}})\mu(B_{i_1 i_2 \dots i_{k-1} i_k})/M_{i_1 i_2 \dots i_{k-1}} \\
< &\  \mu(B_{i_1 i_2 \dots i_{k-1} i_k}) \delta^{-kq}\mu(B_{0})^{-1},
\end{split}
\end{equation}
where we used~\eqref{eq:N kei} and the recursive assumption that 
\[
\nu(B_{i_1 i_2 \dots i_{k-1}}) < \mu(B_{i_1 i_2 \dots i_{k-1}}) \delta^{-(k-1)q}\mu(B_{0})^{-1}.
\] 
This concludes the
general step of the construction.

Next, we define
\[
\wtilde E=\bigcap_{k=1}^\infty \bigcup_{i_1,\dots,i_k} B_{i_1 i_2 \dots i_{k}},
\]
so that $\wtilde E\sub E\cap B_0$ is a non-empty compact set
(here we need the assumptions that $X$ is complete and $E$ is closed;
recall also that balls are assumed to be closed). 
Using the Carath\'eodory construction (cf.\ e.g.\ \cite[pp.\ 54--55]{mat})
for the set function $\nu$, we obtain a Borel regular measure $\tilde\nu$ 
which is supported on $\wtilde E$ and satisfies
$\tilde\nu(B_{i_1 i_2 \dots i_{k}})=\nu(B_{i_1 i_2 \dots i_{k}})$
for all of the balls in the construction (see also \cite[pp.\ 13-14]{F}).

If $x\in \wtilde E$ and $0<r<R$, we choose
$k\in\N$ such that $R \delta^{k}= r_k \leq r < R \delta^{k-1}$.
Then there exists a constant $C_1>0$ 
(depending on the doubling constant and $\delta$) 
such that
$B(x,r)$ intersects at most $C_1$ of the balls
$B_{i_1 i_2 \dots i_{k}}$ from the $k$:th step of the construction;
let these be $B'_1,\dots,B'_N$. These balls are pairwise disjoint,
contained in $B(x,3r)$, and they
cover $\wtilde E\cap B(x,r)$, and thus we have by~\eqref{eq:nupallo},
the choice of $k$,
and doubling that
\begin{equation}\label{eq: nu of ball}
\begin{split}
\tilde\nu(B(x,r)) & = \sum_{j=1}^N\tilde \nu(B'_j) = \sum_{j=1}^N\nu(B'_j) 
\le \sum_{j=1}^N \mu(B'_j) \delta^{-kq}\mu(B_{0})^{-1} \\
& \le C \mu(B(x,r))(r/R)^{-q}\mu(B_{0})^{-1}.
\end{split}
\end{equation}

Finally, let $\{B(z_i,r_i)\}_i$ be a cover of $E\cap B_0$ with
balls of radii $0<r_i<R$.
Using~\eqref{eq: nu of ball}, 
we conclude that
\[
1 = \tilde \nu\big(E\cap B_0\big)\leq \sum_{i} \tilde \nu(B(z_i,r_i))
\leq C\sum_{i}\frac {\mu(B(z_i,r_i))r_i^{-q}}{\mu(B_{0}) R^{-q}},
\]
and so taking the infimum over all such covers yields
\[
\Ha^{\mu,q}_{R}\big(E \cap B_0\big)\ge  C\,R^{-q}\mu(B_0),
\]
as desired.
\end{proof}

Consequently, we obtain a characterization for the upper Assouad codimension
(of closed sets) in terms of Hausdorff content density:

\begin{coro}\label{coro:h char of ucodima}
 Let $E\sub X$ be a closed set. Then $\ucodima(E)$ is the infimum of all $q\ge 0$
 for which there exists $C\ge 0$ such that~\eqref{eq: hausd} holds for every $w\in E$ 
 and all $0<R<\diam(E)$.
\end{coro}

\begin{proof} 
Let $q\ge 0$ be such that~\eqref{eq: hausd} holds for every $w\in E$ and all $0<R<\diam(E)$,  
and let $\{B_i\}$ be a maximal packing of $E \cap B(w,R/2)$ with balls of radius $0<r<R$. 
Then $\{2B_i\}$ is a cover of $E \cap B(w,R/2)$, 
and so the doubling condition and~\eqref{eq: hausd} imply
\[
r^{-q}\mu(E_r\cap B(w,R)) \ge c r^{-q} \sum_i\mu(B_i) \ge c (2r)^{-q}\sum_i\mu(2B_i) \ge c R^{-q}\mu(B(w,R)).
\]
Thus $\ucodima(E)$ gives a lower bound for exponents satisfying~\eqref{eq: hausd}.

On the other hand, Lemma~\ref{lemma:a to h} shows that there can not be a
larger lower bound for these $q$, and thus $\ucodima(E)$ is the infimum, 
as was required.
\end{proof}

\begin{proof}[Proof of Theorem~\ref{thm:main fat}]
We assumed that $\ucodima(\Omega^c)<p-\beta$ and $p-\beta>1$, and thus
we can choose $q>1$ so that $\ucodima(\Omega^c)<q<p-\beta$.
Lemma~\ref{lemma:a to h} then implies that 
\begin{equation}\label{eq:hausd}
\Ha^{\mu,q}_{R}\big(\Omega^c \cap B(w,R)\big)\ge  C\,R^{-q}\mu(B(w,R))
\end{equation}
for every $w\in \Omega^c$ and all $0<R<\diam(\Omega^c)$,
with a constant $C>0$ independent of $w$ and $R$.
By~\cite[Thm.~4.1]{LPAMS2}, this condition is sufficient for $\Omega$ 
to admit a $(p,\beta)$-Hardy inequality, as desired. 
The following remarks are however in order here: 

The condition in~\cite[Thm.~4.1]{LPAMS2}
actually requires that
\begin{equation}\label{eq:hausd alt}
\Ha^{\mu,q}_{\dom(x)}\big(\bdry\Omega \cap B(x,2\dom(x))\big)\ge  
C\,\dom(x)^{-q}\mu\big(B(x,2\dom(x))\big)
\end{equation}
for all $x\in\Omega$, where $\dom(x)=\dist(x,\bdry\Omega)$.
Recall from Section~\ref{sect:dim} that
the validity of a Poincar\'e inequality implies that $X$
is quasiconvex, and thus 
$\dist(x,\bdry\Omega)\simeq \dist(x,\Omega^c)$, and so the
different distance function 
causes no problems here. Moreover,
inspecting the proofs in~\cite{LPAMS2} one sees that 
for $\beta\le 0$ the assumption~\eqref{eq:hausd alt}
can be replaced in Lemma~3.1(a) of~\cite{LPAMS2} with
the condition~\eqref{eq:hausd}. One subtlety here is the
case when $\Omega$ is unbounded, since then~\eqref{eq:hausd}
is needed for all radii $0<R<\infty$, 
and thus we have to assume that in this case
$\Omega^c$ is unbounded as well.
Once Lemma~3.1(a) of~\cite{LPAMS2} is established, the
$(p,\beta)$-inequalities for $\beta>0$ follow just like in the
proof of~\cite[Thm.~4.1]{LPAMS2}; the idea is the same as in
the proof of Theorem~\ref{thm:main} of the present paper.

Let us also remark that the proofs in~\cite{LPAMS2} require the
validity of a $p_0$-Poincar\'e inequality for $1\le p_0<p$,
which is guaranteed by the self-improvement result of Keith and Zhong~\cite{KeZ}.
\end{proof}

\begin{remark}
Actually both 
the \emph{inner boundary density} of~\eqref{eq:hausd alt}
and the complement density~\eqref{eq:hausd},
with an exponent $1\le q<p$, are equivalent to
the \emph{uniform $p$-fatness} of
$\Omega^c$, see~\cite{KLT}.
The deep fact that (also) uniform fatness is a
self-improving condition (see~\cite{Lewis,BMS}) 
is essential in the necessity part of this claim. 
\end{remark}

\section{Necessary conditions}\label{sect:nec}

In this section we extend the previously known necessary conditions for Hardy inequalities
to cover also weighted Hardy inequalities in metric measure spaces. In~\cite{KZ} and~\cite{LT},
such conditions were obtained in the unweighted case $\beta=0$ in metric spaces,
and in~\cite{LMM} for weighted inequalities in the Euclidean setting.
Let us mention here that actually no Poincar\'e inequalities are
needed to establish the results in this section, so we only need to
assume that $\mu$ is doubling.

We have the following generalization of~\cite[Thm.~1.1]{LMM} 
and~\cite[Thm.~6.1]{LT}: 

\begin{thm}\label{thm:dich for bdry}
Let $1\le p<\infty$ and $\beta\neq p$, and
assume that $\Omega\sub X$ admits a $(p,\beta)$-Hardy
inequality. Then there exists $\eps>0$, depending only on the given data,
such that either
\[\codimh(\Omega^c)<p-\beta-\eps\quad\text{ or }\quad
 \lcodima(\Omega^c)>p-\beta+\eps.\]
In particular, $\codimh(\Omega^c)<p-\beta$ or $\lcodima(\Omega^c)>p-\beta$.
\end{thm}

Here the ``given data'' means the parameters $p$ and $\beta$
and the constants in the doubling condition and in the
assumed Hardy inequality. Our next result gives a local version of such a dimension dichotomy; see~\cite[Thm.~5.3]{LMM} for the Euclidean case
and~\cite[Thm.~6.2]{LT} for the case $\beta=0$.

\begin{thm}\label{thm:KZ}
Let $1\le p<\infty$ and $\beta \neq p$, and
assume that $\Omega\sub X$ admits a $(p,\beta)$-Hardy inequality.
Then there exists $\eps>0$, 
depending only on the given data,
such that for each ball $B_0\sub X$ either
\[
\codimh(2B_0\cap \Omega^c)<p-\beta-\eps
\]
or 
the Aikawa condition~\eqref{eq:aikawa} holds with an exponent $q>p-\beta+\varepsilon$ for all $w\in \Omega^c\cap B_0$ and all $0<r<\rad(B_0)$.
\end{thm}

Here the factor $2$ in $2B_0$ is not essential 
(but convenient), any fixed $L>1$ can be used instead.

\begin{remark}
We can not in general conclude in Theorem~\ref{thm:KZ} that either $\codimh(2B_0\cap \Omega^c)<p-\beta-\varepsilon$ or $\lcodima(B_0\cap \Omega^c)>p-\beta+\varepsilon$, since the latter would require the Aikawa condition for all $0<r<\diam(X)$, and this we can not reach under the assumptions of the theorem. Nevertheless, if we further assume that there is a constant $C>0$ and an exponent $s>p-\beta$ such that
\begin{equation}\label{eq:meas_bound}
\frac{\mu(B(x,r))}{\mu(B(x,R))}\le C\Bigl(\frac r R\Bigr)^s
\end{equation}
for all $x\in X$ and all $0<r<R<\diam(X)$, then it is possible to conclude in the setting of Theorem~\ref{thm:KZ} that either $\codimh(2B_0\cap \Omega^c)<p-\beta-\varepsilon$ or $\lcodima(B_0\cap \Omega^c)>p-\beta+\varepsilon$.
The main idea here is that the Aikawa condition, for all $0<r<\rad(B_0)$, implies that also the condition in the definition of the lower Assouad codimension holds for all $0<r<R<\rad(B_0)$ with some exponent $t>p-\beta+\varepsilon$ (cf.~Remark~\ref{rmk:LT}), while for other radii $0<r<R<\diam(X)$ the latter condition follows with the help of the above relative measure bound~\eqref{eq:meas_bound} (possibly with another $\varepsilon>0$). Notice, in particular, that \eqref{eq:meas_bound} holds in a $Q$-regular space for $s=Q$. 
\end{remark}

Recall that our sufficient conditions for Hardy inequalities
were given in terms of 
$\lcodima(\Omega^c)$ and $\ucodima(\Omega^c)$.
However, in the above necessary conditions 
it is not possible to replace $\codimh$ by the (larger)
$\ucodima$ in either of the theorems, cf.\ the discussion in
Section~\ref{sect:sharp examples}.
Also the assumption $\beta\neq p$ is essential in both of the theorems,
as the result need not hold for the $(p,p)$-Hardy inequality, see~\cite{LMM}.
On the other hand, for $\beta>p$ the claims reduce to trivialities, since always
$\lcodima(E)\ge 0$.

One important ingredient in the proofs of these
necessary conditions is the following self-improvement result for Hardy inequalities.

\begin{proposition}\label{prop:hardy si}
Let $1 \le p<\infty$ and $\beta\in\R$,
and assume that $\Omega\sub X$ admits a $(p,\beta)$-Hardy inequality. 
Then there exists $\eps>0$, 
depending only on the given data,
such that $\Omega$ admits $(p,\tilde\beta)$-Hardy inequalities
whenever 
$\beta-\eps\le\tilde\beta\le\beta+\eps$.
Moreover, the constant $C>0$ in all these Hardy 
inequalities can be chosen to
be independent of the particular $\tilde\beta$. 
\end{proposition}

The proof of Proposition~\ref{prop:hardy si}
is almost identical to the Euclidean case, which follows from
the case $s=0$ of~\cite[Lemma~2.1]{lesi},
so we omit the details. 
Notice in addition that while~\cite[Lemma~2.1]{lesi} is formulated
only for $1<p<\infty$, the same proof actually works also when $p=1$.

Another fact that we need in the proofs of Theorems~\ref{thm:dich for bdry}
and~\ref{thm:KZ} is that
if a part of the complement of $\Omega$
is small enough,  
then the test functions for the Hardy inequalities 
need not vanish in that particular part of $\Omega^c$.

\begin{lemma}\label{lemma:over bdry}
Let $1\le p<\infty$ and $0\le \beta < p$, and
assume that $\Omega\sub X$ admits a $(p,\beta)$-Hardy inequality
with a constant $C_0>0$.
Assume further that  $U\sub X$ is an open set
such that
\begin{equation}\label{eq: compl is small}
\Ha_{\diam(U)}^{\mu,p-\beta}(U\cap \Omega^c)=0.
\end{equation}
Then a $(p,\beta)$-Hardy inequality holds for all
$u\in \Lip_0(\Omega\cup U)$
with a constant $C_1=C_1(C_0,p)>0$. 
\end{lemma}

\begin{proof}
Let $u\in \Lip_0(\Omega\cup U)$ with an upper gradient $g_u$.
By the definition of $\Ha_{\diam(U)}^{\mu,p-\beta}$,
there then exist, for a fixed $j\in\N$, 
balls $B_i^j=B(w_i,r_i)$ with $w_i\in\spt(u)\cap U\cap \Omega^c$ and $r_i\le \diam(U)$, $i=1,\dots,N_j$, so that
$\spt(u)\cap U\cap \Omega^c \sub\bigcup_{i=1}^{N_j} B_i^j$ and 
\begin{equation}\label{eq:Hsum}
\sum_{i=1}^{N_j} \mu(B_i^j){r_i}^{-p+\beta}
\leq \|u\|_\infty^{-p} 2^{-j}.
\end{equation}

Let $B_0=B(x_0,R_0)$ be a ball such that $\spt(u)\cap U\cap \Omega^c \subset \frac 1 2 B_0$. 
Iteration of the doubling condition shows that then there exists $Q>0$
and a constant $C>0$ such that
$\mu(B_i^j)/\mu(B_0)\ge C (r_i/R_0)^Q$ for all $i$ and $j$; see for instance~\cite[Lemma~3.3]{BB}. 
Moreover, since one can always choose
a larger $Q$ in this condition, we may assume that $Q>p-\beta$. Thus it follows 
from~\eqref{eq:Hsum} that, for each $j\in\N$,
all the radii $r_i$ (of the balls $B_i^j$) satisfy $r_i^{Q-p+\beta} \le C 2^{-j}$, where the constant $C>0$ may depend on
$u$ and $B_0$, but is independent of $j$. In particular, $r_i\to 0$ uniformly as $j\to\infty$,
and hence we may in addition assume that the covers $\{B_i^j\}_{i=1}^{N_j}$ are nested, i.e.,
$\bigcup_{i=1}^{N_{j+1}} B_i^{j+1} \subset \bigcup_{i=1}^{N_{j}} B_i^{j}$ for all $j\in\N$.

We now define cut-off functions
$\psi_j(x)=\min_i\{1,{r_i}^{-1}d(x,2B_i^j)\}$.
Each function $\psi_j$ has an upper gradient $g_{\psi_j}$
satisfying
$g_{\psi_j}^p\le\sum_i r_i^{-p}\Char{3B_i^j}$
(cf.~\cite[Cor.~2.20]{BB}).
Set $u_j=\psi_j u$.
Then $u_j\in\Lip_0(\Omega)$, and 
$g_{u_j} = g_{\psi_j}|u|+g_u$ 
is an upper gradient of $u_j$ (cf.~\cite[Thm.~2.15]{BB}).
In addition, since the covers were assumed to be nested and 
$r_i\to 0$ uniformly as $j\to\infty$, we have that
$u_j\leq u_{j+1}$ for each $j\in\N$ and $u_j\to u$ pointwise in $\Omega$.

Since 
$\beta\ge 0$, we have $\dom(y)^\beta\le r_i^{\beta}$
for all $y\in B_i^j$, and thus
the $(p,\beta)$-Hardy inequality for the
functions $u_j$ and estimate~\eqref{eq:Hsum} imply that
\[\begin{split}
\int_{\Omega} |u_j|^p \dom^{\beta-p}\,d\mu & 
     \le C  \left[\|u\|^p_\infty\int_\Omega g_{\psi_j}^p\dom^\beta\,d\mu 
        + \int_\Omega g_u^p\dom^\beta\,d\mu\right]\\
  &  \le  C\left[ 
      \|u\|^p_\infty \sum_{i=1}^{N_j} \mu(B_i^j){r_i}^{-p+\beta} + \int_\Omega g_u^p\dom^\beta\,d\mu\right]\\
  &  \le  C 2^{-j} + C \int_{\Omega} g_u^p\dom^\beta\,d\mu,
  \end{split}
\]
where $C=C(C_0,p)>0$.
The claim now follows by monotone convergence,
since $u_j(x)\to u(x)$ in $\Omega$.
\end{proof}

Let us record here the following consequence 
of the previous lemma, which gives an improvement to
Theorem~\ref{thm:main}:

\begin{coro}\label{coro:nonzero bdry}
Let $1\le p <\infty$ and $\beta<p-1$,
and assume that $X$ and $\Omega$ are as in Theorem~\ref{thm:main},
in particular that $\lcodima(\Omega^c)>p-\beta$.
Then a $(p,\beta)$-Hardy inequality holds for
all $u\in\Lip_b(\Omega)$.
\end{coro}

\begin{proof}
For $\beta<0$ the claim follows directly from 
Proposition~\ref{prop:hardy and assouad beta<0}.
For $\beta\ge 0$ 
we have by Theorem~\ref{thm:main} that $\Omega$ admits
a $(p,\beta)$-Hardy inequality.
We now choose $U=(\Omega^c)_1=\{x\in X : d(x,\Omega^c)<1\}$.
Since $p-\beta<\lcodima(\Omega^c)\le \codimh(\Omega^c)$ and $\Omega^c\sub U$,
it follows in particular that
$\Ha_{\diam(U)}^{\mu,p-\beta}(U\cap \Omega^c)=0$.
Thus Lemma~\ref{lemma:over bdry}
implies that actually the 
$(p,\beta)$-Hardy inequality
holds for all $\Lip_0(\Omega\cup U) = \Lip_b(X)$,
and the claim follows.
\end{proof}

Lemma~\ref{lemma:over bdry} 
and the self-improvement of the Aikawa condition from 
Lemma~\ref{lemma:aikawa si} now yield the 
following result, which is essentially
a ``weighted'' version of~\cite[Lemma~2.4]{KZ}. 
For Euclidean spaces, a similar result can be found in~\cite[Lemma~5.2]{LMM},
but note that there the proof is different and especially avoids
the use of Gehring's Lemma, thus making the proof therein 
more self-contained. The approach of~\cite{LMM} 
could be used in the present setting of metric spaces as well, 
but we chose instead to follow the
outline of the proofs from~\cite{KZ} for the sake of brevity and
also to emphasize the role of 
the self-improvement result of Lemma~\ref{lemma:aikawa si}.

\begin{lemma}\label{lemma:int dist}
Let $1\le p<\infty$, $0 \le \beta < p$, and
assume that $\Omega\sub X$ admits a $(p,\beta)$-Hardy inequality.
Assume further that $B_0=B(x_0,R)\sub X$ is an open ball
such that
$\Ha_R^{\mu,p-\beta}(2B_0\cap \Omega^c)=0.$
Then there exists $\delta>0$, depending only on the given data,
such that the Aikawa condition~\eqref{eq:aikawa} holds 
with the exponent $q=p-\beta+\delta$ for all $w\in \Omega^c\cap B_0$ and all $0<r<R$.
\end{lemma}

\begin{proof}
Let $w\in\Omega^c\cap B_0$ and $0<r<R/2$, and
denote $U=2B_0$ and $B=B(w,r)$, so that $2 B\sub 2B_0$. 
Define
$\vphi(x)=r^{-1} d\big(x,X\setminus 2B\big)$. 
Then $\vphi$ is a Lipschitz function with
a compact support in $\Omega\cup U$,
$\vphi\ge 1$ in $B$, and $g_\vphi = r^{-1}\Char{2B}$ 
is an upper gradient of $\vphi$.
By Lemma~\ref{lemma:over bdry} the $(p,\beta)$-Hardy 
inequality holds for $\vphi$, and since $\dom\le 2r$ in $2B$
and $\beta\ge 0$,
we obtain
\[
\int_B \dist(y,\Omega^c\cap B_0)^{\beta-p}\,d\mu(y)\leq \int_{2B} \vphi^p\dom^{\beta-p}\,d\mu
\leq C_1\int_{2B}g_\vphi^p\dom^\beta\,d\mu\leq C_2 \mu(B)r^{-p+\beta}.
\]

In particular, the Aikawa condition~\eqref{eq:aikawa} holds with $q=p-\beta>0$ (for $R/2\le r < R$ the claim follows
by covering $B(x,r)$ with smaller balls).
By Lemma~\ref{lemma:aikawa si} there then exists $\delta>0$ such
that the Aikawa condition holds with $p-\beta+\delta$, 
proving the claim.
\end{proof}

\begin{remark}\label{rmk:indep delta}
Since $\delta>0$ in Lemma~\ref{lemma:int dist}
depends only on the data associated to $X$, $\Omega$,
and the $(p,\beta)$-Hardy inequality, we have the following
uniformity result:
If $(q,\beta)$-Hardy inequalities hold for all
$p_1<q<p_2$ with a constant $C_1$,
we can choose $\delta>0$ in Lemma~\ref{lemma:int dist} to be independent of the particular $q$; 
more precisely, then $\delta=\delta(p_1,p_2,\beta,C_1,\Omega,X)>0$.
\end{remark}

We have now established enough tools to prove Theorem~\ref{thm:KZ}.
The proof follows the lines of the proofs of \cite[Corollary~2.7]{KZ} 
and~\cite[Thm.~6.2]{LT}, but we present the main ideas here for the
convenience of the reader.

\begin{proof}[Proof of Theorem~\ref{thm:KZ}]
Let $B_0=B(x_0,R)\sub X$.
It is clear that if $\beta>p$, we choose $\delta<\beta-p$ and then the Aikawa condition~\eqref{eq:aikawa}
holds with the exponent $q=p-\beta+\delta<0$, and so we only need to consider the case $\beta < p$.
First of all, 
we may assume that $\beta\ge 0$. Indeed,
if this is not the case, we have,
by~\cite[Thm.~2.2]{lesi}, that 
$\Omega$ admits a $(p-\beta,0)$-Hardy inequality
with a constant depending only on the data, 
and now we may consider this instead of the original $(p,\beta)$-Hardy inequality.
Notice that even though~\cite[Thm.~2.2]{lesi} is written in Euclidean spaces,
the proof applies almost verbatim in metric spaces.

By the self-improvement of Hardy inequalities from Proposition~\ref{prop:hardy si},
we find $\eps_1>0$ and $C_1>0$ 
such that $\Omega$ admits $(p,\tilde\beta)$-Hardy inequalities for all
$\beta\le\tilde\beta\le \beta+\eps_1$, and moreover the constant in all these inequalities 
can be taken to be $C_1$. In addition, we require that $\eps_1\leq p-\beta$.

Let then $0<\eps<\eps_1/2$ to be specified later. If
$\codimh(2B_0\cap \Omega^c)<p-\beta-\eps$,
the claim holds, and thus we may assume that
$\codimh(2B_0\cap \Omega^c)\ge p-\beta-\eps$.
It follows that
\[
\Ha_R^{\mu,q}(2B_0\cap \Omega^c)=0
\quad\text{ for } q = p-\beta-2\eps.
\]
As $\Omega$ admits a $(p,\beta+2\eps)$-Hardy inequality
and $p\ge\beta+\eps_1 > \beta + 2\eps$, 
we may use Lemma~\ref{lemma:int dist} 
to conclude that there exists 
$\delta>0$, independent of the particular choice of $\eps<\eps_1/2$ 
(cf.\ Remark~\ref{rmk:indep delta}), 
such that the Aikawa condition~\eqref{eq:aikawa} holds with the exponent
$q=p-(\beta+2\varepsilon)+\delta=p-\beta-2\varepsilon+\delta$. We now choose
$\eps<\min\{\eps_1/2,\delta/3\}$, and the claim follows.
\end{proof}

The global dimension dichotomy in
Theorem~\ref{thm:dich for bdry} follows along the same lines as above:
If $\Omega$ admits a $(p,\beta)$-Hardy inequality for $0\le \beta < p$, and
if in addition
$\codimh(\Omega^c)\ge p-\beta-\eps$, we
obtain from Lemma~\ref{lemma:int dist} that
\[
\int_{B} \dom(x)^{-p+\beta+2\eps-\delta}\,dx\leq C \mu(B)r^{-p+\beta+2\eps-\delta}
\]
for \emph{any} ball $B=B(w,r)$ with
$w\in\Omega^c$ and $0<r<\diam(X)$, where $C$ and $\delta$ are independent of $B$
and the particular $\eps$.
Choosing $\varepsilon>0$ as in the proof of Theorem~\ref{thm:KZ} shows that the Aikawa condition~\eqref{eq:aikawa} holds with an exponent
$q>p-\beta+\varepsilon$ for all $w\in\Omega^c$ and all $0<r<\diam(X)$. Hence we conclude from Remark~\ref{rmk:LT} that indeed $\lcodima(\Omega^c)>p-\beta+\varepsilon$.

\section{Combining thick and thin parts of the complement}\label{sect:combo}

Theorem~\ref{thm:main} gives a sufficient condition for Hardy inequalities in the case where the 
complement of $\Omega$ is thin. Conversely,
Theorem~\ref{thm:main fat} gives such a condition in the case where the complement is thick 
(everywhere and at all scales). Nevertheless, requiring the whole complement to be either 
thick or thin rules out all cases where the complement contains both
large and small pieces;
an easy (and well-understood) example is the punctured ball $B(0,1)\setminus\{0\}\subset\mathbb{R}^n$. 
In the next proposition we 
show how it is possible to combine the results of Theorems~\ref{thm:main} and~\ref{thm:main fat}
for this kind of domains.
A slightly different approach to Hardy inequalities in such domains, for $\beta=0$, 
was given in~\cite[Section~5]{LS}, and
in the Euclidean case earlier results for weighted inequalities 
can be found in~\cite{LMM}; both of these require
additional accessibility properties for $\Omega$.
On the other hand, the results from~\cite{LMM} also cover the case $\beta\ge p-1$, 
where such extra conditions are known to be indispensable
(cf.\ Section~\ref{sect:sharp examples}).

\begin{proposition}\label{prop:thick n thin}
Let $1 < p<\infty$ and $\beta<p-1$. 
If $\beta\le 0$, we assume that $X$ supports a $p$-Poincar\'e inequality, and
if $\beta > 0$ we assume that $X$ supports a $(p-\beta)$-Poincar\'e inequality.
Let $\Om_0\sub X$ be an open set satisfying $\ucodima(\Omega^c)<p-\beta$.
If $F\subset\ol\Om_0$ is a closed set with $\lcodima(F)>p-\beta$,
then $\Om=\Om_0\setminus F$ admits a $(p,\beta)$-Hardy inequality.

Moreover, a $(p,\beta)$-Hardy inequality (in $\Omega$) actually holds for all
$u\in\Lip_0(\Omega_0)$, i.e.\ the test functions need not vanish in $F\cap\Omega_0$.
\end{proposition}

Since this result (in this generality) is new even in 
Euclidean spaces, let us formulate this special case as a corollary:

\begin{coro}\label{coro:thick rn thin}
Let $1 < p<\infty$ and $\beta<p-1$, and
assume that $\Om_0\sub \R^n$ 
is an open set satisfying $\ldima(\Omega^c)>n-p+\beta$.
If $F\subset\ol\Om_0$ is a closed set with $\udima(F)<n-p+\beta$,
a $(p,\beta)$-Hardy inequality holds 
in $\Omega=\Omega_0\setminus F$ for all
$u\in\Lip_0(\Omega_0)$.
\end{coro}

\begin{proof}[Proof of Proposition~\ref{prop:thick n thin}]
Again, it suffices to prove the claim for $\beta<0$, with a suitable control for the
constant, since then the claim for $\beta\ge 0$ follows just as in the proof of
Theorem~\ref{thm:main} from Section~\ref{sect:main}.
For $\beta<0$, the idea is to modify the proof of Proposition~\ref{prop:hardy and assouad beta<0}.
In the present case we can no longer ``chain to infinity''
as in \eqref{eq: chain}, but we can instead
use for sufficiently large balls centered at $F$ the fact that a $(p,\beta)$-Hardy inequality 
holds in $\Omega_0$.

Let $\W(\Omega_0)$ be a Whitney-type cover of $\Omega_0$ with balls
$B(x,c\dom(x))$, $x\in\Omega$, where $0<c<1/2$ is such that
the balls $4\lambda B$, $B\in\W(\Omega)$, have a uniformly bounded overlap
(see e.g.~\cite{BBS}); here $\lambda$ is the dilatation constant from the Poincar\'e
inequality.
Let $B_1,B_2,\dots\in\W(\Omega_0)$ be such that $F\cap\Omega_0\sub\bigcup_{i}B_i$.  

Without loss of generality, let us first consider $E:=B_1\cap F$, where $B_1=B(x,R)$.
Choose $k_0\in\Z$ so that $4\cdot2^{k_0} < R \le 4\cdot2^{k_0+1}$, 
and let $\B_k$, $N_k$, and $A_k$ for $k\le k_0$ be just as in the proof of 
Proposition~\ref{prop:hardy and assouad beta<0}
for this set $E$.
Then $4B\sub 4B_1$ for all $B\in\B_{k_0}$. 

Now let $u\in\Lip_0(\Omega)$.
We divide each $\B_k$, $k\le k_0$,
into two subsets as follows: We set for $B\in\B_k$ that $B\in\B_k^{(1)}$ if 
$|u_{4B_1}|\le \frac 1 2 |u_{4B}|$, and otherwise $B\in\B_k^{(2)}$.
For convenience, we also set $\B_{k_0+1}=\B_{k_0+1}^{(1)}=\B_{k_0+1}^{(2)}:=\{B_1\}$.
For $B\in\B_k^{(1)}$, $k\le k_0$, we then have
\[
|u_{4B}|\le |u_{4B}-u_{4B_1}|+|u_{4B_1}|\le |u_{4B}-u_{4B_1}|+\tfrac 1 2 |u_{4B}|,
\]
and thus
\[
\tfrac 1 2 |u_{4B}|\le \sum_{j=k}^{k_0-1}|u_{4B^j}-u_{4B^{j+1}}|+|u_{4B^{k_0}}-u_{4B_1}|,
\]
where $B=B^k$.
This is analogous to estimate~\eqref{eq: chain}, and, indeed, the balls $B\in\B_k^{(1)}$
can be treated just like in the proof of Proposition~\ref{prop:hardy and assouad beta<0}, yielding 
as in~\eqref{eq:2nd sum} and~\eqref{eq:2nd sum final} 
that
\begin{equation}\label{eq:2nd sum final TNT}
\begin{split}
 \sum_{k=-\infty}^{k_0+1} 2^{k(\beta-p)} \sum_{B\in\B_k^{(1)}}\int_{4B} |u_{4B}|^p\,d\mu 
& \le C \int_{4\lambda B_1} g_u^p \dom^\beta\,d\mu.
\end{split}
\end{equation}

On the other hand, for the balls $B\in\B_k^{(2)}$, $k\le k_0$, we have
by H\"older's inequality that
\[
|u_{4B}| <  2 |u_{4B_1}| \le  2\biggl(\vint_{4B_1}|u|^p\,d\mu\biggr)^{1/p},
\]
and since $d_{\Omega_0}(y)\simeq R$ for all $y\in 4B_1$, we obtain 
\begin{equation}\label{eq:to hardy}
|u_{4B}|^p \le C \frac{R^{p-\beta}}{\mu(B_1)}\int_{4B_1}|u|^p {d_{\Omega_0}}^{\beta-p}\,d\mu.
\end{equation}
Now pick $q$ such that $\lcodima(F)>q>p-\beta$, whence by definition
\[
\frac{\mu(E_{2^{k}})}{\mu(B_1)} \le C \frac{2^{kq}}{R^{q}}\quad \text{ for all } k\le k_0+1.
\]
Estimate~\eqref{eq:to hardy} and the bounded overlap of the balls $4B$, for $B\in\B_k^{(2)}$
with a fixed $k$, then imply
\begin{equation}\label{eq:2nd sum final TNT2}
\begin{split}
 \sum_{k=-\infty}^{k_0+1} 2^{k(\beta-p)} & \sum_{B\in\B_k^{(2)}}\int_{4B} |u_{4B}|^p\,d\mu 
 \le C \int_{4B_1}|u|^p {d_{\Omega_0}}^{\beta-p}\,d\mu 
    \sum_{k=-\infty}^{k_0+1} \frac{2^{k(\beta-p)}}{R^{\beta-p}}\frac{\mu(E_{2^{k}})}{\mu(B_1)}\\
 & \le C \int_{4B_1}|u|^p {d_{\Omega_0}}^{\beta-p}\,d\mu 
    \sum_{k=-\infty}^{k_0+1} \frac{2^{k(q+\beta-p)}}{R^{q+\beta-p}} 
    \le C \int_{4B_1}|u|^p {d_{\Omega_0}}^{\beta-p}\,d\mu,
\end{split}
\end{equation}
since $q+\beta-p>0$.

Following the proof of Proposition~\ref{prop:hardy and assouad beta<0},  
we can now combine~\eqref{eq:2nd sum final TNT} and~\eqref{eq:2nd sum final TNT2},
and we obtain as in~\eqref{eq:compute} that
\begin{equation}\label{eq:compute2}\begin{split}
 \int_{4B_1} & |u|^p  \dom^{\beta-p}\,d\mu  
       \le C \sum_{k=-\infty}^{k_0+1} \int_{A_k} |u|^p\dom^{\beta-p}\,d\mu\\
 & = \sum_{k=-\infty}^{k_0+1} 2^{k(\beta-p)} \sum_{B\in\B_k}\int_{4B} |u-u_{4B}|^p\,d\mu 
     + C \sum_{k=-\infty}^{k_0+1} 2^{k(\beta-p)} \sum_{B\in\B_k}\int_{4B} |u_{4B}|^p\,d\mu\\
 &  \le C \int_{4\lambda B_1} g_u^p \dom^\beta\,d\mu +  
       C \int_{4B_1}|u|^p {d_{\Omega_0}}^{\beta-p}\,d\mu.
\end{split}
\end{equation}
Note that here $A_{k_0+1}=4B_1\setminus N_{k_0}$,
and that the first integral in the second line can be estimated just like
in~\eqref{eq:first sum}.

Similar estimates hold of course for all balls $B_1,B_2,\dots$,
with constants independent of $i$.
Moreover, if $x\in \Omega\setminus\bigcup_i 4B_i$, then
$\dom(x)\ge C d_{\Omega_0}(x)$: If $\dom(x) = d_{\Omega_0}(x)$
the claim is trivial, so we may assume that
$\dom(x)=d(x,w)$ for some $w\in F$. Pick $B_i\ni w$. Since
$x\notin 4B_i$ and $\rad(B_i)\ge \tilde c d(w,\Omega_0)$ with $0<\tilde c<1$, 
it follows that
\[
\dom(x)=d(x,w)\ge \rad(B_i)\ge \tilde c d(w,\Omega_0)
\geq \tilde c d(x,\Omega_0) - \tilde cd(x,w)= \tilde c d(x,\Omega_0) - \tilde c\dom(x),
\]
and the claim follows. In particular $\dom(x)^{\beta-p}\le C d_{\Omega_0}(x)^{\beta-p}$
for these $x$, and thus estimate~\eqref{eq:compute2} for each $B_i$, the $(p,\beta)$-Hardy 
inequality for $\Omega_0$, and
the bounded overlap of the balls $4\lambda B_i$ (and thus of $4B_i$) yield
\begin{equation*}
\begin{split}
\int_{\Omega} |u|^p  \dom&^{\beta-p}\,d\mu 
    \le \int_{\Omega\setminus\bigcup_i 4B_i}  |u|^p  \dom^{\beta-p}\,d\mu 
       + \sum_{i} \int_{4B_i}  |u|^p  \dom^{\beta-p}\,d\mu\\
  & \le C \int_{\Omega}  |u|^p  {d_{\Omega_0}}^{\beta-p}\,d\mu 
       + C \sum_{i}\int_{4\lambda B_i} g_u^p \dom^\beta\,d\mu +  
       C \sum_{i} \int_{4B_i}|u|^p {d_{\Omega_0}}^{\beta-p}\,d\mu\\
  & \le C \int_{\Omega_0} |u|^p  {d_{\Omega_0}}^{\beta-p}\,d\mu 
       + C \sum_{i}\int_{4\lambda B_i} g_u^p \dom^\beta\,d\mu\\
  & \le C \int_{\Omega_0} g_u^p  {d_{\Omega_0}}^{\beta}\,d\mu 
       + C  \int_{\bigcup_{i}4\lambda B_i} g_u^p \dom^\beta\,d\mu\\
  & \le C  \int_{\Omega_0} g_u^p{d_{\Omega_0}}^\beta\,d\mu 
       + C   \int_{\Omega} g_u^p \dom^\beta\,d\mu 
    \le C  \int_{\Omega} g_u^p \dom^\beta\,d\mu;
\end{split}
\end{equation*}
we also used the facts that $d_{\Omega_0}(x)^\beta\le\dom(x)^\beta$ 
for all $x\in\Omega$ (since $\beta<0$) and that $\mu(F)=0$.

The above constant $C>0$ naturally depends on $\beta$, but for $-1<\beta<0$ close enough to $0$,
the dependence can be reduced again to the form $C=|\beta|^{-1}C^*$, where $C^*>0$ is independent of $\beta$;
the details are exactly the same as in the proof of Proposition~\ref{prop:hardy and assouad beta<0}. 
Hence $(p,\beta)$-Hardy inequalities, for $0\le \beta < p-1$, now follow
along the same lines as in the proof of the corresponding case of Theorem~\ref{thm:main} 
(cf.\ Section~\ref{sect:main}). 

Finally, regarding the boundary values, we see that in the above proof of the case $\beta<0$ 
it is not necessary for $u$ to vanish in $F$, and thus this case indeed holds for all
$u\in\Lip_0(\Omega_0)$. On the other hand, in the case $\beta\ge 0$ we can apply
Lemma~\ref{lemma:over bdry} with $U=\Omega_0$
(since $p-\beta<\lcodima(F)\le\codimh(\Omega_0\cap\Omega^c)$), and it follows again that
a $(p,\beta)$-Hardy inequality holds for all
$u\in\Lip_0(\Omega\cup\Omega_0)=\Lip_0(\Omega_0)$.
This concludes the proof.
\end{proof}

\section{Sharpness of the results}\label{sect:sharp examples}

We close the paper with an examination of the sharpness of our results.
In particular, we consider the necessity of the assumptions in our main
theorems.

It was already mentioned in the introduction that the bound $p-\beta$ is very natural for 
the dimensions in all of the sufficient and necessary conditions, and can not be
improved. Moreover, the bound $p-\beta$ for the lower Assouad codimension appears both in the
sufficient and necessary conditions in the cases where the complement is thin, and so it is obvious 
that $\lcodima$ is the optimal concept of dimension in this setting. However, 
when (a part of) the complement is thick, the sufficient conditions are given
in terms of the upper Assouad codimension, while in the necessary conditions
the Hausdorff codimension is used, and so these conditions do not quite meet. 
This raises the question whether it could be possible to improve the bounds in
either sufficient or necessary conditions by using a different
concept of dimension. Since the possibility to combine
thick and thin parts in the sufficient conditions (Proposition~\ref{prop:thick n thin})
immediately rules out such improvements in the global results, the sharpness of
the conditions (in terms of dimensions) is established if we show that
(i) $\codimh(2B_0\cap\Omega^c)$ can not be replaced by
$\ucodima(2B_0\cap\Omega^c)$ in Theorem~\ref{thm:KZ}, and
that on the other hand 
(ii) the local bound $\codimh(B\cap\Omega^c)<p-\beta$,
for all balls $B=B(w,r)$ with $w\in\Omega^c$,
does not suffice for the $(p,\beta)$-Hardy inequality in $\Omega$.

The following construction yields a counterexample for both (i) and (ii).

\begin{example}\label{ex:ex}
 We consider here, for simplicity, the unweighted case $\beta=0$ in $\R^n$,
 with $\mu$ being the usual Lebesgue measure. Then $|\nabla u|$ is an
 optimal upper gradient for each Lipschitz function $u$.

 Denote $w_j=(2^{-j},0,\dots,0)\in\R^n$ for $j\in\N$ and $0=(0,0,\dots,0)\in\R^n$, and
 let $F=\{w_j:j\ge 2\}\cup\{0\}\sub\R^n$ and
 $\Omega_1=B(0,2)\setminus F\sub\R^n$. Since $\ucodima(B(0,2)^c)=0$ and
 $\lcodima(F)=n$, we have by Proposition~\ref{prop:thick n thin}
 that $\Omega_1$ admits a $p$-Hardy inequality whenever
 $1<p<n$ (and $\Omega_1$ does not admit an $n$-Hardy inequality). 
 
 Next, we replace each point $w_j$ with a ball $B_j=B(w_j,2^{-2j})$, that is,
 we consider the domain
 $\Omega=B(0,2)\setminus\bigl(\bigcup_{j\ge 2}B_j \cup \{0\}\bigr)$.
 Then also $\Omega$ admits $p$-Hardy inequalities whenever
 $1<p<n$; this can be seen as follows:
 
 Write $\Omega'=B(0,2)\setminus\bigl(\bigcup_{j\ge 2}2B_j \cup \{0\}\bigr)$
 and $\Omega''=\bigcup_{j\ge 2} 2B_j\setminus B_j$, so that $\Omega=\Omega'\cup\Omega''$.
 For all $x\in\Omega'\sub\Omega_1$ we have 
 $d_{\Omega_1}(x)\ge d_{\Omega}(x) \ge \frac 1 2 d_{\Omega_1}(x)$, and thus, 
 using the $p$-Hardy inequality in $\Omega_1$, we have for all
 $u\in\Lip_0(\Omega)\sub \Lip_0(\Omega_1)$ that
 \begin{equation}\label{eq:omega'}
  \int_{\Omega'}|u|^p \dom^{-p}\,dx 
  \le C \int_{\Omega_1}|u|^p {d_{\Omega_1}}^{-p}\,dx \le 
  C \int_{\Omega_1} |\nabla u|^p \,dx 
  = C \int_{\Omega} |\nabla u|^p \,dx.
 \end{equation}
 On the other hand, for all $x\in\Omega''$ the complement of $\Omega$ near $x$
 satisfies the density condition $\Ha^{n}(\Omega^c\cap B(x,2\dom(x)))\ge C \dom(x)^n$, 
 and so it follows that even a stronger pointwise
 $1$-Hardy inequality holds for these points (cf.~\cite{haj,KLT}): If $u\in\Lip_0(\Omega)$
 and $x\in\Omega''$, then
 $|u(x)|\le C\dom(x) M_{2\dom(x)}|\nabla u|(x)$, where $M_{2\dom(x)}$ is the
 restricted Hardy--Littlewood maximal operator. 
 The $L^p$-boundedness of $M_{2\dom(x)}$ then yields
 \begin{equation}\label{eq:omega''}
  \int_{\Omega''}|u(x)|^p \dom(x)^{-p}\,dx 
  \le \int_{\Omega''} \bigl(M_{2\dom(x)}|\nabla u|(x)\bigr)^p \,dx \le 
  C \int_{\Omega} |\nabla u(x)|^p \,dx,
 \end{equation}
 and so the $p$-Hardy inequality for $\Omega$, for every $1<p<n$, 
 follows by combining~\eqref{eq:omega'} and~\eqref{eq:omega''}.

 Nevertheless, while it is clear that 
 $\codimh(\Omega^c\cap B(w,r))=\lcodima(\Omega^c\cap B(w,r))=0$
 whenever $w\in\Omega^c$, and thus in particular
 $\codimh(\Omega^c\cap B(w,r))<p$ in accordance with Theorem~\ref{thm:KZ},
 this example shows that $\codimh$ cannot be replaced with $\ucodima$ in the theorem:
 Indeed, we see that for all balls $B$ centered at the origin, 
 e.g.\ for $B=B(0,1/2)$, we have
 $\ucodima(\Omega^c\cap B)=n$ (i.e.\ $\ldima(\Omega^c\cap B)=0$), since for all $j\ge 3$ the
 set $\Omega^c\cap 2^{j}B_j\sub B$ can be covered by the ball $B_j$,
 and here the ratio of the radii of the balls has no positive lower bound.
 Hence neither of the estimates $\ucodima(\Omega^c\cap 2B)<p$ and 
 $\lcodima(\Omega^c\cap B)>p$ holds here when $1<p<n$, even though $\Omega$
 admits a $p$-Hardy inequality. We conclude that in general, the bound $\codimh(2B_0\cap \Omega^c)<p-\beta$ is optimal in Theorem~\ref{thm:KZ}, and clearly the same conclusion holds for Theorem~\ref{thm:dich for bdry} as well. Thus point~(i) is established.

 For point~(ii), we notice that $\Omega^c$ is not \emph{uniformly
 perfect}, that is, there are relatively large annuli around the
 balls $B_j$ which do not intersect $\Omega^c$. Since uniform perfectness
 of $\Omega^c$ is equivalent to the validity of an $n$-Hardy inequality in 
 $\Omega\sub \R^n$ (see~\cite{KS}), we conclude that $\Omega$ does not
 admit an $n$-Hardy inequality.
 On the other hand, $\codimh(\Omega^c\cap B(w,r))=0<n$ whenever $w\in\Omega^c$ and $r>0$,
 and so this example shows
 that the uniformity provided by the upper Assouad codimension in the
 sufficient conditions of Theorem~\ref{thm:main fat} and 
 Proposition~\ref{prop:thick n thin} is essential,
 and can not be replaced with the condition that $\codimh(\Omega^c\cap B)<p-\beta$
 for all balls centered at $\Omega^c$. This yields point~(ii).
\end{example}

Let us next take a look at
the requirement $p-\beta>1$ in Theorems~\ref{thm:main}
and~\ref{thm:main fat}.
It was already mentioned in the Introduction that
for Theorem~\ref{thm:main fat}, the unit ball $B=B(0,1)\sub\R^n$ 
shows the necessity of this condition,
since $\ucodima(\R^n\setminus B) = 0$, 
but $B$ admits $(p,\beta)$-Hardy inequalities only when $p-\beta>1$.
In fact, it is now understood that in the case $p-\beta\le 1$ it is
the thickness of the boundary (rather than the complement) that
plays a role in Hardy inequalities. For instance, the planar domain $\Omega$
bounded by the usual von Koch -showflake curve of dimension
$\lambda=\log 4/\log 3$ admits a $(p,\beta)$-Hardy inequality
if (and only if) $\beta<p-2+\lambda$, i.e.\ exactly when
$p-\beta>\ucodima(\bdry\Omega)$ (cf.~\cite{kole}).
However, the requirement $\ucodima(\bdry\Omega)<p-\beta$ alone 
is not sufficient for a $(p,\beta)$-Hardy inequality if $\ucodima(\bdry\Omega)<1$,
as is shown by~\cite[Examples~7.3 and~7.4]{kole}, but
certain accessibility conditions are required in addition;
cf.~\cite{kole,LPAMS2} and see also Remark~\ref{rmk:direct}.

In the case of Theorem~\ref{thm:main},
the unbounded domain $\wtilde\Omega^s$ indicated at the end 
of~\cite[Example~6.3]{LMM} serves as an example
where $1=\lcodima(\R^n\setminus\wtilde\Omega^s)$, but the domain does not 
admit any $(p,\beta)$-Hardy inequalities
when $p-\beta \le 1$. 
Nevertheless, all known counterexamples here are such that
$\lcodima(\Omega^c)\le 1$ (and thus in the examples in $\R^n$ 
we have $\dima(\Omega^c)\ge n-1$),
and so it could be asked if the requirement $p-\beta>1$ could actually be
removed (or weakened) if $\lcodima(\Omega^c) > 1$.
Under additional accessibility conditions 
Hardy inequalities can be obtained 
in the range $p-\beta\le 1$ in the case of thin 
boundaries as well; see~\cite{LMM} for the Euclidean case. 

Finally, the unboundedness of $X$ in Theorem~\ref{thm:main} can not be relaxed either, 
as the simple example $X=[-1,1]^2$, $\Omega=X\setminus\{0\}$ shows. Namely, here we can
consider functions $u_j\in\Lip_0(\Omega)$ 
which have value one in $X\setminus B(0,2^{-j})$ and $|\nabla u_j|\simeq 2^j$ 
in $B(0,2^{-j})\setminus B(0,2^{-j-1})$ (and $|\nabla u_j|$ vanishes elsewhere).
These functions show that $(p,\beta)$-Hardy
inequalities fail whenever $p-\beta\le n = \lcodima(\Omega^c)$.
On the other hand, if $\Omega=\R^2\setminus B(0,1)$,
then $\ucodima(\Omega^c)=0$ but $\Omega$ does not admit a $(2,0)$-Hardy inequality,
and this shows that it is  
essential in Theorem~\ref{thm:main fat} that the complement 
of an unbounded $\Omega$ is unbounded as well.

\subsection*{Acknowledgements}

Part of the research leading to this work was conducted when the author was visiting 
the Department of Mathematical Sciences of the University of Cincinnati in February 2013. 
The author wishes to express his gratitude to Nageswari Shanmugalingam for fruitful 
discussions, and the whole department for their hospitality. 
Thanks are also due to Antti K\"aenm\"aki for discussions related to Section~\ref{sect:fat},
and to the anonymous referee for many useful comments.

\end{document}